\documentclass{amsart}
\usepackage{hyperref,enumerate}
\usepackage{xcolor}

\newcommand{\bl}{\tilde{\lambda}_2}
\newcommand{\tf}{\tilde{\lambda}_1}

\newtheorem{theorem}{Theorem}[section]
\newtheorem{lemma}[theorem]{Lemma}
\newtheorem{proposition}[theorem]{Proposition}
\newtheorem{corollary}[theorem]{Corollary}
\newtheorem{claim}[theorem]{Claim}

\theoremstyle{definition}
\newtheorem{definition}[theorem]{Definition}

\theoremstyle{remark}
\newtheorem{remark}[theorem]{Remark}
\numberwithin{equation}{section}

\title[Variational Properties]{Variational properties of the second eigenvalue of the conformal laplacian}

\author{Matthew J. Gursky}
\address{Department of Mathematics \\
         University of Notre Dame\\
         Notre Dame, IN 46556}

\author{Samuel P\'{e}rez-Ayala}
\address{Department of Mathematics \\
         University of Notre Dame\\
         Notre Dame, IN 46556}

\begin{document}

\begin{abstract}
Let $(M^n,g)$ be a closed Riemannian manifold of dimension $n\ge 3$. Assume $[g]$ is a conformal class for which the Conformal Laplacian $L_g$ has at least two negative eigenvalues. We show the existence of a (generalized) metric that maximizes the second eigenvalue of $L_g$ over all conformal metrics (the first eigenvalue is maximized by the Yamabe metric).  We also show that a maximal metric defines either a nodal solution of the Yamabe equation, or a harmonic map to a sphere.  Moreover, we construct examples of each possibility.
\end{abstract}

\thanks{The first author acknowledges the support of NSF grant DMS-1811034 and DMS-1547292}

\maketitle

\section{Introduction}\label{Introduction}

Let $(M^n,g)$ be a closed Riemannian manifold of dimension $n\ge 3$, and let $L_g$ denote the Conformal Laplacian, defined by
\begin{equation}
L_g := -\Delta_g + c_nR_g.
\end{equation}
Here $c_n := \frac{n-2}{4(n-1)}$, and $\Delta_g$ is the negative Laplace-Beltrami operator. This operator is conformally invariant in the following sense: if $ g_u := u^{\frac{4}{n-2}}g\in [g]$, then for any $\phi\in C^\infty(M^n)$ it holds that
\begin{equation}\label{Conf-Invar}
L_{g_u}(\phi) = u^{-\frac{n+2}{n-2}}L_g(u\phi).
\end{equation}
This formula implies the following conformal invariance of the Dirichlet energy:
\begin{align} \label{cid}
\int_M \phi L_{g_u} \phi \, dv_{g_u} = \int_M (u \phi) \, L_g (u \phi)\, dv_g.
\end{align}
Therefore, the Rayleigh quotient satisfies
\begin{align} \label{Ruu} \begin{split}
\dfrac{ \int_M \phi L_{g_u} \phi \, dv_{g_u} }{ \int_M \phi^2 \, dv_{g_u} } &= \dfrac{ \int_M (u \phi) \, L_g (u \phi)\, dv_g}{ \int_M (u \phi)^2 u^{\frac{4}{n-2}} \, dv_g} \\
&= \dfrac{ \int_M \psi \, L_g \psi \, dv_g}{ \int_M \psi^2 u^{\frac{4}{n-2}} \, dv_g } \\
&=: \mathcal{R}^u_g(\psi),
\end{split}
\end{align}
where $\psi = u \phi$.

Since $M^n$ is compact, the spectrum $\text{Spec}(L_g)$ of $L_g$ is discrete, and we denote it by
\begin{equation}
\lambda_1(L_g) < \lambda_2(L_g)\le \cdots \le \lambda_k(L_g) \nearrow \infty,
\end{equation}
where the eigenvalues are repeated according to their multiplicities. If we fix a choice of a conformal representative $g\in[g]$, then given $g_u = u^{\frac{4}{n-2}}g$ we denote
\begin{align} \label{lug}
\lambda_k(u) := \lambda_k(L_{g_u}).
\end{align}
In view of the conformal invariance of the Rayleigh quotient in (\ref{Ruu}), the min-max characterization of the $k^{th}$-eigenvalue can be expressed as
\begin{align} \label{Lku}
\lambda_k(u) = \inf_{\Sigma_k \subset W^{1,2}(M^n,g) } \sup_{\psi \in \Sigma_k \setminus \{ 0 \}} \mathcal{R}^u_g(\psi),
\end{align}
where $\Sigma_k \subset W^{1,2}:=W^{1,2}(M^n,g)$ denotes a $k$-dimensional subspace of $W^{1,2}$.

The conformal invariance of $L_g$ also implies the conformal invariance of various spectrally-defined quantities:
\begin{enumerate}[(i)]
\item The sign of $\lambda_1(L_g)$ is a conformal invariant (see \cite{Kazdan}), and agrees with the
sign of the Yamabe invariant
\begin{align*}
Y(M^n,[g]) := \inf_{u \in W^{1,2} \setminus \{ 0 \}} \dfrac{ \int_M u \, L_g u \, dv_g }{ \left( \int_M |u|^{\frac{2n}{n-2}} \, dv_g \right)^{\frac{n-2}{n}}}.
\end{align*}

\item The dimension of $\ker L_g$ is a conformal invariant.  This is immediate from (\ref{Conf-Invar}).

\item The number of negative eigenvalues of $L_g$, $\nu([g])$, is also a conformal invariant, and its size is not topologically obstructed (see \cite{Canzani}).
\end{enumerate}

Our main focus in this work will be the existence of an extremal for the normalized eigenvalue functional $F_k$ defined by
\begin{equation}\label{E-Functional}
u  \longmapsto F_k(u) := \lambda_k(u)\left(\int_M u^{\frac{2n}{n-2}}\;dv_g\right)^{\frac{2}{n}}.
\end{equation}
As pointed out in \cite{Ammann}, if the Yamabe invariant $Y(M^n,[g]) \geq 0$, then
\begin{align} \label{infY}
\inf_{ g_u \in [g]} F_1(u) = Y(M^n,[g]).
\end{align}
When $Y(M^n,[g]) < 0$ the same argument shows
\begin{align} \label{supY}
\sup_{ g_u \in [g]} F_1(u) = Y(M^n,[g]).
\end{align}
In particular, by the resolution of the Yamabe problem, a minimizer (respectively, a maximizer) for $F_1(u)$ exists if $Y(M^n,[g]) \geq 0$ (resp., $< 0$).
We will therefore be interested in the variational properties of $F_k$ when $k\geq 2$.

For conformal classes with $Y(M^n,[g]) \ge 0$, Ammann-Humbert \cite{Ammann} called
\begin{align} \label{infY2}
\mu_2(M^n,[g]) := \inf_{ g_u \in [g]} F_2(u)
\end{align}
the {\em second Yamabe invariant}.  Under certain assumptions they were able to prove the existence of a ``generalized metric'' attaining the second Yamabe invariant; i.e., a metric of the form $g_u = u^{\frac{4}{n-2}}g$, with
\begin{align} \label{LNdef}
u \in L^{\frac{2n}{n-2}}_{+}(M^n,g) := \{ u \in L^{\frac{2n}{n-2}}(M^n,g)\, :\, u \geq 0 \ a.e. \}\setminus \{ 0 \}.
\end{align}
As we explain below, the existence of minimizers for $F_2$ is related to the existence of nodal solutions for the Yamabe equation.  We should also point out that in the case of positive Yamabe invariant, the {\em supremum} of $F_k$ is always infinity for any $k\ge1$ (see \cite{Ammann2}).

We will explain in Section \ref{SetUp} how one can define $\lambda_k(u)$ when $u$ is not necessarily smooth and positive -- indeed, $u$ can even vanish on a set of positive measure.  However, when the spectrum of $L_g$ includes negative eigenvalues, to avoid the possibility that $\lambda_k(u) = -\infty$ we will need to be more restrictive.  Therefore, we define
\[
L^{\frac{2n}{n-2}}_{>0}(M^n,g):=\{u\in L^{\frac{2n}{n-2}}_+(M^n,g): u^{-1}(0) \text{ has zero Riemannian measure}\}.
\]

If the Yamabe invariant $Y(M^n,[g]) < 0$ and $\nu([g]) = 1$, then $\lambda_2(u) \geq 0$ for all $g_u \in [g]$. In this case, El Sayed \cite{ElSayed} showed that the approach of Ammann-Humbert still works, and the second Yamabe invariant is attained under certain assumptions.  However, if $\nu([g]) > 1$, then it is not difficult to see that
\begin{align} \label{inf2neg}
\inf_{ g_u \in [g]} F_2(u) = -\infty.
\end{align}
Indeed, (\ref{supY}) suggests that in this setting one should seek to {\em maximize} $F_2$ in $[g]$.  Our main result is that this is always possible, provided $\ker L_g$ is trivial:

\begin{theorem}\label{MainTheorem}
Let $(M^n,g)$ be a closed Riemannian manifold equipped with a conformal class $[g]$  satisfying $\nu([g])>1$ and $0\not \in \text{Spec}(L_g)$. Then there is a $\bar{u} \in L^{\frac{2n}{n-2}}(M)$ that maximizes the normalized eigenvalue functional
\begin{equation}
F_2 : u\in L^{\frac{2n}{n-2}}_{>0}(M^n,g) \longmapsto \lambda_2(u)\left(\int_M u^{\frac{2n}{n-2}}\;dv_g\right)^{\frac{2}{n}}.
\end{equation} 
Also, $\bar u\in C^{\mu_0}(M^n)\cap C^\infty(M^n\setminus\{\bar u=0\})$, $\mu_0 \in (0,1)$.

Moreover, if $\bar{u} \in L^{\frac{2n}{n-2}}_{>0}(M^n,g)$ is any maximizer, then there exists a collection $\{\bar \phi^\alpha\}_{i=1}^k\in C^{2,\mu_0}(M^n)$ of second generalized eigenfunctions (see Section \ref{SetUp})   satisfying
\begin{equation}\label{limitEulerEq}
\bar u^2 - \sum_{\alpha=1}^k \bar (\phi^\alpha)^2 =0.
\end{equation}
Here $1 \leq k \leq \dim E_2(\bar{u})$, where $E_2(\bar{u})$ is the space of generalized eigenfunctions corresponding to $\lambda_2(\bar{u})$.
\end{theorem}
 
 As a consequence of Theorem \ref{MainTheorem}, in each conformal class satisfying the assumptions of the theorem, there is either a nodal solution of the Yamabe equation, or a harmonic map into a sphere, depending on the integer $k$ in the conclusion:

\begin{corollary}\label{NodalHarmonic}
Let $\bar u\in C^{\mu_0}(M^n)\cap C^\infty(M^n\setminus\{\bar u=0\})$ be a  maximal function provided by Theorem \ref{MainTheorem}. We have the following two cases:
\begin{enumerate}
\item If $k=1$, then $\bar u=|\bar \phi|$ on $M^n$, and $\bar \phi$ is a nodal solution of
\begin{equation}
L_g\bar \phi = \lambda_2(\bar u)  |\bar \phi|^{\frac{4}{n-2}} \, \bar \phi.
\end{equation}
\item If $k>1$, then the map
\begin{equation}\label{Harmonic}
\bar U:=(\bar \phi^1/\bar u,\cdots, \bar \phi^k/\bar u): (M^n\setminus\{\bar u=0\}, \bar u^{\frac{4}{n-2}}g) \longrightarrow (\mathbb{S}^{k-1},g_{\text{round}})
\end{equation}
 defines a harmonic map.
\end{enumerate}
\end{corollary}

As we now explain, it is possible to construct examples for both cases.  For the first case, observe that by definition
\begin{align*}
1 \leq k \leq \nu([g]).
\end{align*}
In particular, if $\nu([g]) =2$, then $k=1$ (since $\lambda_1(\bar u)$ must be simple by Lemma \ref{efProp}).  Moreover, Corollary \ref{NodalHarmonic} implies that any maximal function $\bar u$ of the functional $F_2$ on $L^{\frac{2n}{n-2}}_{>0}(M^n,g)$ is of the form $\bar u = |\bar \phi|$ for $\bar \phi\in E_2(\bar u)$.  Since $\bar \phi$ changes sign, this means that $\bar u$ cannot be strictly positive and consequently there is no bona fide Riemannian metric in $[g]$ maximizing $F_2$.

By contrast, in Section \ref{Example} we will construct explicit examples for which $k \ge 2$ and the maximal metric induces a harmonic map:

\begin{theorem} \label{IntroExample}  Let $(H,h)$ be a closed Riemannian manifold with constant negative scalar curvature, suitably normalized (see Section \ref{Example}).  Then the product metric $(M,g) = (H \times S^1, h + d\theta^2)$ is maximal in its conformal class.  In particular, eigenfunctions $\{ \psi_1, \psi_2 \}$ for the Laplacian on the $S^1$-factor are eigenfunctions for $\lambda_2(L_g)$, and define a harmonic map
\begin{align*}
\Psi = (\psi_1, \psi_2) : M \rightarrow S^1,
\end{align*}
given by projection onto the $S^1$-factor.
\end{theorem}

\medskip

\noindent {\bf Remarks.}  \bigskip \begin{enumerate}

\item In the work of Ammann-Humbert, if $g_{\underline u} \in [g]$ is a (generalized) metric that minimizes $F_2$, then $\lambda_2(\underline{u})$ is always simple\footnote{The simplicity of a minimizer in this case can also be shown by a first variation argument; see Remark \ref{Simplicity}.} (see Theorem 3.4 of \cite{Ammann}).  When $k > 1$ in Corollary \ref{NodalHarmonic}, then $\lambda_2(\bar u)$ is not necessarily simple; for example, the multiplicity of $\lambda_2(L_g)$ for the product metric $g$ on $H \times S^1$ is two.   \medskip

\item In fact, it is easy to adapt the proof of Theorem \ref{IntroExample} to construct examples with $k$ arbitrarily large, by taking products with spheres of high enough dimension. \medskip

\item Theorem \ref{IntroExample} also shows that maximal metrics can be smooth.  This is another significant contrast to the problem of minimizing $F_2$: in the work of Ammann-Humbert, if $g_{\underline u} \in [g]$ is a (generalized) metric minimizing $F_2$, then $\underline{u} = |\phi|$, where $\phi$ is a nodal solution of the Yamabe equation.  In particular, $g_{\underline u}$ is never a smooth Riemannian metric.   \medskip

\item Another surprising aspect of Theorem \ref{IntroExample} is that the product metric is a Yamabe metric, hence (as we observed above) is simultaneously maximal for the first normalized eigenvalue functional.  This is remarkably different from the case of the Laplace operator on surfaces, where it is known that metrics cannot maximize consecutive eigenvalues (\cite{ElSoufi2}).  \medskip

\end{enumerate}

\medskip 

It would be interesting to explore what happens for higher order eigenvalues, assuming they are negative, that is, for $\lambda_3(L_g)\le \cdots\le \lambda_{\nu([g])}(L_g)$ for a conformal class with $\nu([g])\ge 3$. However, in the case of surfaces, a bubbling phenomena was first noticed by Nadirashvilli in \cite{Nadirashvilli2}. This complicates the existence theory for maximal metrics of higher order eigenvalues on closed surfaces, and a spectral gap assumption is necessary (see \cite{Nadirashvilli4, Petrides2}).

On surfaces, the connection of maximal eigenvalues to harmonic maps is well known (see \cite{Nadirashvilli}, \cite{ElSoufi}).  If $\Sigma$ is a closed Riemannian surface, then Petrides \cite{Petrides} and Nadirashvilli-Sire \cite{Nadirashvilli3} have shown that for any conformal class $[g]$ on $\Sigma$, there exists a metric $\bar g \in [g]$, which is smooth except possibly at finitely many points, such that
\begin{equation}\label{ConformalEigenvalue}
\lambda_1(-\Delta_{\bar{g}}) \, \text{Area}(\Sigma,\bar{g}) = \Lambda_1(\Sigma,[g]) := \sup_{\tilde g \in [g]} \lambda_1(-\Delta_{\tilde g}) \, \text{Area}(\Sigma, \tilde g)
\end{equation}
(see also work by Kokarev in \cite{Kokarev}). In addition, to this maximal metric there is an associated harmonic map into a higher dimensional sphere (\cite{Fraser}).  In a similar spirit, Fraser-Schoen have shown the connection between extremal properties of Steklov eigenvalues on surfaces with boundary and free boundary minimal surfaces. To our knowledge, the relation between extremal properties of the spectrum of the Conformal Laplacian and harmonic maps is new.

The outline of the paper is as follows. In Section \ref{SetUp} we make the necessary definitions to define the maximization problem.  In particular, we define the $k^{th}$-eigenvalue associated to a function $u$ in $L^{\frac{2n}{n-2}}_+(M^n,g)$. Also, we discuss key compactness results for eigenfunctions. Section \ref{VariationFormulas} is devoted to the derivation of the first variation formulas for the normalized eigenvalue functional $F_2$. The main result we prove is that along certain deformations $u_t$ in $L^{\frac{2n}{n-2}}_{>0}(M^n,g)$, both of the one-sided derivatives of $F_2(u_t)$ exist at $t=0$. The techniques we use in this section follow closely those employed by Kokarev in \cite{Kokarev}. In Section \ref{RegularizedFunctional} we introduce a regularized version of the problem.  The main idea behind the regularization is that we can control the size of the zero set of extremal functions. We prove existence of maximizers for the regularized functional and derive an associated Euler Lagrange Equation via classical separation theorems as it was done in \cite{ElSoufi}, \cite{ElSoufi2} for Laplace eigenvalues, and in \cite{Fraser} for Laplace and Steklov eigenvalues.  In Section \ref{Estimates} we obtain uniform estimates for the sequence of maximizers, and in Section \ref{TakingLimit} we prove Theorem \ref{MainTheorem} and Corollary \ref{NodalHarmonic}. Finally, in Section \ref{Example} we provide an example showing that the second case in Corollary \ref{NodalHarmonic} ($k>1$) occurs.

\vskip.2in
\section{Preliminaries} \label{SetUp}

Assume we are given a conformal class $[g]$ for which $\nu([g])>1$ and $0\not \in \text{Spec}(L_g)$. As pointed out in Section \ref{Introduction}, these assumptions are conformally invariant. Since $\nu([g])>1$ implies that $Y(M^n,[g])<0$, we are allowed to select as reference metric some $g\in [g]$ for which $R_g<0$ everywhere on $M$. From now on these are the assumptions that we will be working with.

Some notation and definitions are necessary. For $p \neq 0$, let us denote
\begin{align*}
L^p_+:= L^p_{+}(M^n,g) = \{ u \in L^p(M^n,g)\, :\, u \geq 0 \ a.e.\}\setminus \{0\}.
\end{align*}
The collection of functions in $L^p_+$ with the property that its zero locus has zero Riemannian measure will be denoted by $L^p_{\tiny>0}$:
\begin{equation}
L^p_{>0}:=\{u\in L^p_+: u^{-1}(0) \text{ has zero Riemannian measure}\}
\end{equation}
We denote by $N:= 2^*=\frac{2n}{n-2}$ the critical Sobolev exponent. Given $u \in C^{\infty}(M^n)$ with $u>0$, $g_u$ denotes the conformal metric $g_u = u^{N-2}g \in [g]$.  As discussed in Section \ref{Introduction}, by conformal invariance and the min-max characterization, we have
\begin{align} \label{Lku}
\lambda_k(u):=\lambda_k(L_{g_u}) = \inf_{\Sigma_k \subset W^{1,2} } \sup_{\psi \in \Sigma_k \setminus \{ 0 \}} \mathcal{R}^u_g(\psi).
\end{align}

This characterization motivates the definition of the $k^{th}$-eigenvalue $\lambda_k(u)$ associated to an arbitrary function $u\in L^{N}_{+}$. However, since the zero set of a function $u\in L^{N}_{+}$ could be large, there is the possibility of having a $k$-dimensional subspace spanned by functions which are linearly independent only on $\{u=0\}$. To rule out this scenario we define $\lambda_k(u)$ via (\ref{Lku}), but we take the infimum over the {\em $k^{th}$ modified Grassmannian}, $Gr_k^u(W^{1,2}):= Gr_k^u(W^{1,2}(M^n,g))$.  More precisely, $\langle \phi_1,\cdots, \phi_k \rangle \in Gr_k^u(W^{1,2})$ if and only if $\{\phi_1,\cdots, \phi_k\}\subset W^{1,2}$ and the functions $\phi_1 u^{\frac{N-2}{2}},\cdots, \phi_k u^{\frac{N-2}{2}}$ are linearly independent (see Section 2.2. in \cite{Ammann}).

\begin{definition}\label{GeneralizedEigenvalue}
For $u\in L^{N}_{+}$, the {\em $k^{th}$-generalized eigenvalue} $\lambda_k(u)$ is defined via
\begin{equation}\label{Lkudef}
\lambda_k(u):=\inf_{\Sigma_k \in Gr_k^u(W^{1,2}) } \sup_{\psi \in \Sigma_k \setminus \{ 0 \}} \mathcal{R}^u_g(\psi)
\end{equation}
The function $u \in L^N_{+}$ is called a {\em generalized conformal factor}, and $g_u=u^{N-2}g$ is called a {\em generalized conformal metric}.
\end{definition}

\noindent {\bf Remarks.}  \begin{enumerate}  \medskip

\item Heuristically, the reason for introducing the spaces $Gr_k^u(W^{1,2})$ is that we are only taking into account what happens on regions where the generalized conformal factor $u \neq 0$.  \medskip

\item The minimax characterization is often used for proving eigenvalue comparison results since the subspaces are independent of any choice of metric.  This is not the case in (\ref{Lkudef}), where the subspaces depend on $u$.  However, we do have the following elementary observation: \medskip

\end{enumerate}

\begin{lemma}\label{TestFunctions}
Let $u,w\in L^N_+$ be generalized conformal factors. If $\{x\in M: u(x) \neq 0\}\subseteq \{x\in M: w(x) \neq 0\}$, then $Gr_k^u(W^{1,2}) \subseteq Gr_k^w(W^{1,2})$.
\end{lemma}

\begin{proof}
It is enough to notice that if the collection $\{\phi_1u^{\frac{N-2}{2}}, \cdots, \phi_ku^{\frac{N-2}{2}}\}$ is linearly independent, then so is $\{\phi_1w^{\frac{N-2}{2}}, \cdots, \phi_kw^{\frac{N-2}{2}}\}$.
\end{proof}

It was observed by El Sayed (see Proposition 2.1 of \cite{ElSayed}) that when the Yamabe invariant $Y(M^m,[g]) < 0$, then it is always possible to construct a generalized conformal factor $u \in L^N_{+}$ such that $\lambda_1(u) = -\infty$.  The following result from the same paper explains why we restrict our conformal factors to $L^N_{>0}$:

\begin{proposition} \label{Prop22}  (Proposition 2.2 of \cite{ElSayed})  If $u \in L^N_{>0}$, then $\lambda_1(u) > -\infty$.
\end{proposition}

We now turn to eigenfunctions.  We say that  $\phi_k \in W^{1,2}(M)$ is a {\em $k^{th}$-generalized eigenfunction} if it is a (non-trivial)  solution of the equation
\begin{equation} \label{genev}
L_g\phi_k = \lambda_k(u)\phi_ku^{N-2}.
\end{equation}
We let $E_k(u)$ denote the space of all $k^{th}$-generalized eigenfunctions.

The following result has significant overlap with Proposition 2.4 in \cite{ElSayed}.  However, some parts of the proof of Proposition 2.4 were either omitted or only sketched, so we provide additional details:

 \begin{proposition} \label{efProp} Assume $\nu([g]) \geq 2$, and let $u \in L^N_{>0}$.  \medskip

\noindent $(i)$  There is a first generalized eigenfunction, i.e., a $W^{1,2}$-solution of
\begin{equation} \label{primo}
L_g\phi_1 = \lambda_1(u)\phi_1 u^{N-2}.
\end{equation}
Moreover, if $\phi \in E_1(u)$ is non-trivial then $\phi \geq 0$ or $\phi \leq 0$, and $\phi^{-1}(0)$ has measure zero. \medskip

\noindent $(ii)$ $\dim E_1(u) = 1$.   \medskip

\noindent $(iii)$  $\lambda_2(u)$ is given by
\begin{equation}\label{Char-e}
 \lambda_2(u) = \inf \mathcal{R}_g^u(\phi),
 \end{equation}
 where the infimum is taken over all functions $\phi\in W^{1,2}$ such that
 \begin{align*}
 \int_M \phi \psi_1 u^{N-2}\;dv_g = 0
 \end{align*}
 for each $\psi_1 \in E_1(u)$.  \medskip

 \noindent $(iv)$  There is a basis $\{ \phi_2^{(1)}, \dots, \phi_2^{(m)} \}$ of $E_2(u)$ satisfying
 \begin{align} \label{ortho2}
 \int_M \phi_2^{(i)} \, \phi_2^{(j)} \, u^{N-2} \, dv_g = 0 \ \ \mbox{if $i \neq j$},
 \end{align}
 and for any $\phi_1 \in E_1(u)$
 \begin{align} \label{ortho1}
 \int_M \phi_2^{(i)} \, \phi_1 \, u^{N-2}\, dv_g = 0, \ \ 1 \leq i \leq m.
 \end{align}
 In particular, there is a $\gamma = \gamma(u) > 0$ such that
 \begin{align} \label{sgap}
 \lambda_2(u) \geq \lambda_1(u) + \gamma.
 \end{align}
 \end{proposition}

\begin{proof}  $(i)$ The existence of a (non-trivial) first eigenfunction is proved in Proposition 2.4 of \cite{ElSayed}.  Let $\psi \in E_1(u)$,
and $\psi^{+}(x) = \max \{ \psi(x), 0\}$ denote the positive part of $\psi$.  We claim that $\psi^{+} \in E_1(u)$ (possibly $\psi^{+} \equiv 0$).  To see this, we use the fact that $\psi$ is a $W^{1,2}$-solution of
\begin{align} \label{E1def}
L_g \psi = \lambda_1(u) \psi \bar{u}^{N-2}.
\end{align}
Using $\psi^{+}$ as a test function, we easily find that $\mathcal{R}^u_g(\psi^{+}) = \lambda_1(u)$.  By the variational characterization of $\lambda_1(u)$, it follows that $\psi^{+} \in E_1(u)$.  Replacing $\psi$ with $-\psi$, we see that either $\psi \geq 0 \, a.e.$ or $\psi \leq 0 \, a.e.$

Suppose $\psi \in E_1(u)$ with $\psi \geq 0 \ a.e.$, $\psi \not\equiv 0$.  We want to show that $\psi^{-1}(0)$ has measure zero.  We first claim that there is a $q_0 > 0$ such that
\begin{align} \label{JNE1}
\int_M \psi^{-q_0} \, dv_g < C_0,
\end{align}
where $C_0 = C_0(\| \psi \|_{q_0}^{-1})$.  To see this, let $B(r) \subset M$ be a geodesic ball of radius $r > 0$, and $\eta \in C^{\infty}(M)$ a cut-off function
 with $\eta \equiv 1$ in $B(r)$, $\eta \equiv 0$ on $M \setminus B(2r)$, and $|\nabla \eta | \leq C r^{-1}$.  Also, for $\epsilon > 0$ let
\begin{align*}
\psi_{\epsilon} = \big( \psi^2 + \epsilon)^{1/2}.
\end{align*}
Then $\eta^2 \psi_{\epsilon}^{-1} \in W^{1,2}$, and using this as a test function in (\ref{E1def}),
\begin{align} \label{JN2}
\int_M  \langle \nabla (\eta^2 \psi_{\epsilon}^{-1}), \nabla \psi \rangle \, dv_g + \int_M c_n R \psi \, (\eta^2 \psi_{\epsilon}^{-1}) \, dv_g = \lambda_1(u) \int_M \psi \, (\eta^2 \psi_{\epsilon}^{-1})  \, u^{N-2} \, dv_g.
\end{align}
Since $\psi \leq \psi_{\epsilon}$, we can expand and estimate in the standard way to obtain
\begin{align*}
\int_{B(r)}  |\nabla \log \psi_{\epsilon}|^2  \, dv_g &\leq \int_M \eta^2 |\nabla \psi_{\epsilon}|^2 \, dv_g \\
&\leq C \int_M \big( \eta^2 + |\nabla \eta|^2 \big) \psi \, \psi_{\epsilon}  \, dv_g + C \int_M \psi \, \psi_{\epsilon} \, \bar{u}^{N-2} \, dv_g \\
&\leq C r^{n-2} + C \int_{B(2r)} u^{N-2} \, dv_g \\
&\leq C r^{n-2} +  C \| u \|_{L^N} r^{n-2}.
\end{align*}
By the Monotone Convergence Theorem it follows that
\begin{align*}
\int_{B(r)} |\nabla \log \psi|^2 \, dv_g \leq C r^{n-2}.
\end{align*}
where $C = C(|\lambda_1(u), \| u \|_N )$.  By the John-Nirenberg inequality, there is a $q_0 > 0$ such that
\begin{align} \label{JN5}
\left( \int_M \psi^{q_0} \, dv_g \right) \left( \int_M \psi^{-q_0} \, dv_g \right) \leq C.
\end{align}
Since we are assuming $\psi \geq 0$ is non-trivial, (\ref{JNE1}) follows.

An immediate consequence of (\ref{JNE1}) is that if $\psi \in E_1(u)$ and $\psi \not\equiv 0$, then $\psi^{-1}(0)$ has measure zero.  \medskip

\noindent $(ii)$  Suppose $\psi_1, \psi_2 \in E_1(u)$, normalized so that
\begin{align*}
\int_M \psi_1 \, dv_g = \int_M \psi_2 \, dv_g = 1.
\end{align*}
Since $\psi = \psi_1 - \psi_2 \in E_1(u)$, by part $(i)$ we see that $\psi \geq 0$ or $\psi \leq 0 \ a.e.$.  Since
\begin{align*}
\int_M \psi \, dv_g = \int_M \psi_1 \, dv_g - \int_M \psi_2 \, dv_g = 0,
\end{align*}
it follows that $\psi = 0 \ a.e.$, hence $\psi_1 = \psi_2$.  This shows that $\dim E_1(u) = 1$.  \medskip

\noindent $(iii)$  Denote
\begin{align} \label{chre}
\tilde \lambda_2(u) = \inf \mathcal{R}_g^u(\phi),
 \end{align}
 where the infimum is taken over all functions $\phi\in W^{1,2}$ such that
 \begin{align*}
 \int_M \phi \psi_1 u^{N-2}\;dv_g = 0
 \end{align*}
 for each $\psi_1 \in E_1(u)$.  Since $\dim E_1(u) = 1$ we can assume the orhtogonality condition holds for some fixed  (non-trivial) $\psi_1 \in E_1(u).$

 For any two-dimensional subspace $\Sigma_2 \in Gr_k^u$, there is a function $\psi_2 \in \Sigma_2$ satisfying the orthogonality condition
\begin{equation}\label{orthogonality}
\int_M \psi_1\psi_2u^{N-2}\;dv_g = 0,
\end{equation}
where $\psi_1$ is the first generalized eigenfunction associated to $\lambda_1(u)$. Then we have
\begin{equation}
\sup_{\phi\in\Sigma_2}\mathcal{R}_g^u(\phi) \ge \mathcal{R}_g^u(\psi_2) \ge \tilde \lambda_2(u),
\end{equation}
where the second inequality follows from the definition of $\tilde \lambda_2(u)$. Taking the infimum over all such $\Sigma_2\in Gr_2^u(W^{1,2})$ yields
\begin{equation}
\lambda_2(u)\ge \tilde \lambda_2(u).
\end{equation}

To deduce the opposite inequality, we start with an arbitrary function $\psi_2\in W^{1,2}$ satisfying (\ref{orthogonality}). Define $\Sigma_2^{\psi_2} := \langle \psi_1,\psi_2\rangle$ and note that it belongs to $Gr_2^u(W^{1,2})$. Therefore,
\begin{equation}
\mathcal{R}_g^u(\psi_2) = \sup_{\phi\in \Sigma_2^{\psi_2}} \mathcal{R}_g^u(\phi) \ge \lambda_2(u),
\end{equation}
where the equality follows from the orthogonality condition (\ref{orthogonality}), while the inequality follows from Definition \ref{GeneralizedEigenvalue}. Taking the infimum over all such $\psi_2$ yields
\begin{equation}
\tilde \lambda_2(u)\ge \lambda_2(u).
\end{equation}
\medskip

\noindent $(iv)$  As observed in \cite{ElSayed}, once the variational characterization of $\lambda_2(u)$ given by part $(iii)$ is established, one can apply the direct variational method to prove the existence of second generalized eigenfunctions.  

\end{proof}

As we observed in the Introduction, the number of negative eigenvalues of $L_g$, which we denote by $\nu([g])$, is a conformal invariant (see \cite{Canzani}).  If $u \in L^N_{+}(M)$ is a generalized conformal
factor, it was claimed in Proposition 3.1 of \cite{ElSayed} that the number of negative generalized eigenvalues associated to $u$ (denoted $\nu(u)$) is equal to $\nu([g])$.  However, the proof is incomplete, as it assumes $u^{-1}(0)$ has zero measure; i.e., that $u \in L^N_{>0}(M)$.   For the record we provide the correct statement and proof:

\begin{proposition}\label{NegativeEigen}
If $u\in L^N_{>0}$, then $\nu(u) = \nu([g])$.
\end{proposition}
\begin{proof}
Given any $k\in\mathbb{N}$, we first prove that $\lambda_k(u)<0$ implies $\lambda_k(1) = \lambda_k(L_g) <0$. If $\lambda_k(1)\ge 0$, then
\begin{equation}
\sup_{\phi\in \Sigma_k} \mathcal{R}_g^1(\phi)\ge 0,
\end{equation}
for all $\Sigma_k\in Gr_k^1(W^{1,2})$. By Lemma \ref{TestFunctions}, $Gr^u_k(W^{1,2})\subseteq Gr_k^1(W^{1,2})$, and thus $\lambda_k(u)\ge 0$. This shows that $\nu([g])\ge \nu(u)$.

Now assume that $u\in L^N_{>0}$. Since the zero locus of $u$ has zero Riemannian measure, by definition, we have $Gr_k^u(W^{1,2}) = Gr_k^1(W^{1,2})$ thanks to Lemma \ref{TestFunctions}. A similar argument shows that $\nu(u)\ge \nu([g])$.
\end{proof}

Controlling the multiplicity of the second eigenvalue is key to the regularization procedure that will be introduced in Section \ref{RegularizedFunctional}.

We will also need a compactness result for generalized eigenfunctions.  Before giving the precise statement, we make some preliminary observations.  Suppose  $\{ u_i \} \subset L^N_{>0}$ is a sequence of generalized conformal factors, normalized so that
\begin{align} \label{normi}
\int_M u_i^N \, dv_g = 1.
\end{align}
If we let
\begin{align*}
v_i = u_i^{N-2} \geq 0,
\end{align*}
then
\begin{align*}
\int_M v_i^{\frac{n}{2}} \, dv_g = \int_M u_i^N \, dv_g = 1.
\end{align*}
Therefore, a subsequence of $\{ v_i \}$ will converge weakly in $L^{\frac{n}{2}}(M)$ to $\bar{v} \in L^{\frac{n}{2}}(M)$, with
\begin{align} \label{vFat}
\int_M \bar{v}^{\frac{n}{2}} \, dv_g \leq 1.
\end{align}
Note that if $\psi \in C^{\infty}(M)$ is non-negative, then
\begin{align*}
0 \leq \int_M \psi v_i \, dv_g \rightarrow \int_M \psi \bar{v} \, dv_g,
\end{align*}
hence $\bar{v} \geq 0$ a.e..  Define $\bar{u} \in L^N_{+}(M)$ by
\begin{align*}
\bar{u}^{N-2} = \bar{v}.
\end{align*}

Of course, it could happen that the sequence of eigenvalues associated to $\{ u_i \}$ degenerates.  The first compactness theorem we prove roughly says that if the sequence of first or second eigenvalues stays bounded, then the corresponding sequence of generalized eigenfunctions is compact: 
\medskip

\begin{proposition} \label{Cprop} Assume $\nu([g]) \geq 2$ and $\ker L_g$ is trivial.

Let $\{ u_i \} \subset L^N_{>0}$ be a sequence of generalized conformal factors, normalized as in (\ref{normi}), and $\bar{u}^{N-2}$ be a weak subsequential limit of $\{ u_i^{N-2} \}$ as in the preamble.  For $k = 1$ or $2$, let $\phi_{k,i}$ be a generalized $k^{th}$ eigenfunction associated to $u_i$, normalized so that
\begin{align} \label{normii}
\int_M \phi_{k,i}^2 u_i^{N-2} \, dv_g = 1.
\end{align}

Assume there is a constant $C = C(g)$ such that
\begin{align} \label{EVB}
|\lambda_k(u_i)| \leq C.
\end{align}
Then there is a $C = C(g)$ such that
\begin{align} \label{WW}
\| \phi_{k,i} \|_{W^{1,2}(M)} + \| \phi_{k,i} \|_{L^{\infty}(M)} \leq C.
\end{align}
After possibly restricting to a subsequence, $\{ \phi_{k,i} \}$ converges weakly in $W^{1,2}(M)$ and strongly in $L^{N/2}(M)$ to $\phi_k \in W^{1,2}(M) \cap L^{\infty}(M)$ satisfying
\begin{align} \label{wE}
L_g \phi_k = \bar{\lambda}_k \phi_k \bar{u}^{N-2}
\end{align}
in the $W^{1,2}$-sense, where $\bar{\lambda}_k = \lim_{i \to \infty} \lambda_k(u_i)$.  Moreover,
\begin{align} \label{L21}
\int_M \phi_k^2 \bar{u}^{N-2} \, dv_g = 1.
\end{align}
\end{proposition}

\medskip 

\begin{remark}  In the statement of Proposition \ref{Cprop}, we do not claim that the (weak) limit of generalized eigenfunctions $\phi_k$ is itself a generalized eigenfunction.  To prove this we need additional conditions on the generalized conformal factor $\bar{u}$; see Proposition \ref{Cpropplus} below. 
\end{remark}

\medskip

\begin{proof} We begin with a key lemma:

\begin{lemma}  \label{L2Lemme} There is a constant $C > 0$ such that for all $i \geq 1$,
\begin{align} \label{L2start}
\| \phi_{k,i} \|_{L^2(M)} \leq C.
\end{align}
\end{lemma}

\begin{proof}   Assume to the contrary that there is a subsequence (still denoted $\{ \phi_{k,i} \}$) such that
\begin{align} \label{bup}
\| \phi_{k,i} \|_{L^2(M)} \rightarrow \infty.
\end{align}
Define
\begin{align} \label{scor1}
\tilde{\phi}_{k,i} = \dfrac{ \phi_{k,i} }{ \| \phi_{k,i} \|_{L^2(M)}}.
\end{align}
Then $\tilde{\phi}_{k,i}$ is a generalized eigenfunction with
\begin{align} \label{newnorm}
\| \tilde{\phi}_{k,i} \|_{L^2(M)} = 1,
\end{align}
and hence is a $W^{1,2}(M)$-solution of
\begin{align} \label{testy}
-\Delta_g \tilde{\phi}_{k,i} + c_n R_g \tilde{\phi}_{k,i} = \lambda_k(u_i) \tilde{\phi}_{k,i} u_{i}^{N-2}.
\end{align}
If we take $\tilde{\phi}_{k,i}$ as a test function in (\ref{testy}), then
\begin{align} \label{testyy} \begin{split}
\int_M \left( |\nabla_g \tilde{\phi}_{k,i}|^2 + c_n R_g \tilde{\phi}_{k,i}^2 \right) \, dv_g &= \lambda_k(u_i) \int_M \tilde{\phi}_{k,i}^2 u_{\epsilon}^{N-2} \, dv_g \\
&\leq 0.
\end{split}
\end{align}
Therefore,
\begin{align*}
\int_M  |\nabla_g \tilde{\phi}_{k,i}|^2 &\leq  -c_n  \int_M R_g \tilde{\phi}_{k,i}^2 \, dv_g  \\
&\leq C  \| \tilde{\phi}_{k,i}\|_{L^2(M)}^2 \\
&\leq C.
\end{align*}
It follows that
\begin{align} \label{W12t}
\| \tilde{\phi}_{k,i} \|_{W^{1,2}(M)} \leq C.
\end{align}
This implies there is a $\tilde{\phi}_k \in W^{1,2}(M)$ such that a subsequence of $\{ \tilde{\phi}_{k,i} \}$ (still denoted $\{ \tilde{\phi}_{k,i} \}$) converges weakly in $W^{1,2}(M)$, and strongly in $L^2(M)$, to $\tilde{\phi}_k$.  By strong convergence,
\begin{align} \label{SSP}
\int_M \tilde{\phi}_k^2 \, dv_g = 1.
\end{align}

Let $\psi \in C^{\infty}(M)$, then using $\psi$ as a test function in (\ref{testy}) we have
\begin{align} \label{testy2}
\int_M \left( \langle \nabla_g \tilde{\phi}_{k,i}, \nabla_g \psi \rangle + c_n R_g \tilde{\phi}_{k,i} \psi \right) \, dv_g = \lambda_k(u_i) \int_M \psi \tilde{\phi}_{k,i} u_i^{N-2} \, dv_g.
\end{align}
We claim that the right-hand side converges to zero as $i \to \infty$.  To see this, we use the fact that $|\lambda_k(u_i)|$ and $|\psi|$ are bounded, hence
\begin{align*}
\left| \lambda_k(u_i) \int_M \psi \tilde{\phi}_{k,i} u_i^{N-2} \, dv_g \right| &\leq C  \int_M |\tilde{\phi}_{k,i} | u_i^{N-2} \, dv_g \\
&= C \dfrac{1}{ \| \phi_{k,i} \|_{L^2(M)}} \int_M  |\phi_{k,i}| u_i^{N-2} \, dv_g \\
&\leq C \dfrac{1}{ \| \phi_{k,i} \|_{L^2(M)}} \left( \int_M  \phi_{k,i}^2 u_i^{N-2} \, dv_g \right)^{1/2} \left( \int_M  u_{\epsilon_i}^{N-2}  \right)^{1/2} \\
&\leq  C \dfrac{1}{ \| \phi_{k,i} \|_{L^2(M)}} \\
&\rightarrow 0, \ \ i \to \infty.
\end{align*}
Also, since $\tilde{\phi}_{k,i}$ converges weakly in $W^{1,2}(M)$, for the left-hand side of (\ref{testy2}) we have
\begin{align} \label{scor2}
\lim_{i \to \infty} \int_M \left( \langle \nabla_g \tilde{\phi}_{k,i}, \nabla_g \psi \rangle  + c_n R_g \tilde{\phi}_{k,i} \psi \right) \, dv_g =  \int_M \left( \langle \nabla_g \tilde{\phi}_k, \nabla_g \psi \rangle  + c_n R_g \tilde{\phi}_k \psi \right) \, dv_g.
\end{align}
Therefore, $\tilde{\phi}_k$ is a $W^{1,2}$-solution of
\begin{align*}
L_g \tilde{\phi}_k = 0.
\end{align*}
Moreover, by (\ref{SSP}), $\tilde{\phi}_k$ is not identically zero.  Since we are assuming that $\ker L_g$ is trivial, this is a contradiction.
\end{proof}

Since the sequence $\{ \phi_{k,i} \}$ is uniformly bounded in $L^2$, we can repeat the arguments of the preceding lemma to obtain a $W^{1,2}$-bound.  More precisely:  for all $i \geq 1$, $\phi_{k,i}$ is a $W^{1,2}$-solution of
\begin{align} \label{Eki}
L_g \phi_{k,i} = \lambda_k(u_i) \phi_{k,i} u_i^{N-2},
\end{align}
and using $\phi_{k,i}$ as a test function in (\ref{Eki}) we get
\begin{align} \label{EnS}
\int_M \big( |\nabla_g \phi_{k,i}|^2 + c_n R_g \phi_{k,i}^2 \big) \, dv_g = \lambda_k(u_i) \int_M \phi_{k,i}^2 u_i^{N-2} \, dv_g \leq 0,
\end{align}
hence
\begin{align*}
\int_M |\nabla_g \phi_{k,i}|^2 \, dv_g \leq C \int_M \phi_{k,i}^2 \, dv_g.
\end{align*}
It follows from (\ref{L2start}) that
\begin{align} \label{EnS2}
\| \phi_{k,i} \|_{W^{1,2}(M)} \leq C,
\end{align}
where $C = C(g)$.

To prove the $L^{\infty}$-bound, let $f_{k,i} := \sqrt{ 1 + \phi^2_{k,i} }$.  Since $\lambda_k(u_i) < 0$, it is easy to see that $f_{k,i}$ satisfies
\begin{align*}
\Delta_g f_{k,i} \geq   c_n R_g f_{k,i}.
\end{align*}
Furthermore, by (\ref{EnS2}), $\| f_{k,i} \|_{W^{1,2}} \leq C$.  It then follows from a standard Moser iteration argument that
\begin{align} \label{pkisup}
\|  \phi_{k,i} \|_{L^{\infty}} \leq \|  f_{k,i} \|_{L^{\infty}} \leq C.
\end{align}
Combining (\ref{EnS2}) and (\ref{pkisup}), we conclude that there is a $\phi_k \in W^{1,2}(M) \cap L^{\infty}(M)$ and a subsequence of $\{ \phi_{k,i} \}$ (still denoted $\{ \phi_{k,i} \}$) that converges weakly to $\phi_k$ in $W^{1,2}(M)$, and strongly in $L^{N/2}(M)$.

To prove (\ref{wE}) and (\ref{L21}) we will need the following elementary result, which will be used repeatedly:

\begin{claim} \label{weakClaim}
If $\{ w_i \} \subset L^{N/2}(M)$ converges strongly to $w$ in $L^{N/2}(M)$, then
\begin{align} \label{www}
\int_M w_i u_i^{N-2} \, dv_g \rightarrow \int_M w \bar{u}^{N-2} \, dv_g, \ \ i \to \infty.
\end{align}
\end{claim}

The proof of (\ref{www}) follows from the fact that $u_i^{N-2}$ converges weakly in $L^{\frac{n}{2}}(M)$ to $\bar{u}^{N-2}$.  Indeed,
\begin{align} \label{www2}
\int_M w_i u_i^{N-2} \,dv_g &= \int_M (w_i - w) u_i^{N-2} \, dv_g + \int_M w u_i^{N-2} \, dv_g.
\end{align}
The first integral on the right converges to zero as $i \to \infty$ by H\"older's inequality:
\begin{align*}
\left| \int_M (w_i - w) u_i^{N-2} \, dv_g \right| \leq \| w_i - w \|_{L^{N/2}(M)} \| u_i \|_{L^N(M)} \rightarrow 0, \ \ i \to \infty.
\end{align*}
Therefore, taking the limit as $i \to \infty$ in (\ref{www2}) and using the weak convergence of $u_i^{N-2}$, the claim follows.

We now show that $\phi_k$ is a $W^{1,2}(M)$-solution of (\ref{wE}).  To see this, let $\psi \in C^{\infty}(M)$.  Since $\phi_{k,i}$ is a weak solution of (\ref{Eki}), it follows that
\begin{align} \label{ww1}
\int_M \left( \langle \nabla_g \psi, \nabla_g \phi_{k,i} \rangle + c_n R_g \psi \phi_{k,i} \right) \, dv_g  = \lambda_k(u_i) \int_M \psi  \phi_{k,i} u_i^{N-2} \, dv_g.
\end{align}
Since $\psi \phi_{k,i}$ converges in $L^{N/2}(M)$ to $\psi \phi_k$, by Claim \ref{weakClaim} we have
\begin{align} \label{ww2}
\lim_{i \to \infty} \lambda_k(u_i) \int_M \psi  \phi_{k,i} u_i^{N-2} \, dv_g = \bar{\lambda}_k \int_M \psi \psi_k \bar{u}^{N-2} \, dv_g.
\end{align}
Also, since $\phi_{k,i}$ converges weakly to $\phi_k$ in $W^{1,2}(M)$, for the left-hand side of (\ref{ww1}) we have
\begin{align} \label{ww3}
\lim_{i \to \infty} \int_M \left( \langle \nabla_g \psi, \nabla_g \phi_{k,i} \rangle + c_n R_g \psi \phi_{k,i} \right) \, dv_g = \int_M \left( \langle \nabla_g \psi, \nabla_g \phi_k \rangle + c_n R_g \psi \phi_k \right) \, dv_g.
\end{align}
Combining (\ref{ww1}) -- (\ref{ww3}), we get (\ref{wE}).

To prove (\ref{L21}), we take $w_i = \phi_{k,i}^2$ and $w = \phi_k^2$ in Claim \ref{weakClaim}.  We need to show $w_i \to w$ in $L^{N/2}(M)$.  By the $L^{\infty}$-bound of $\phi_{k,i}$ and $\phi_k$ we have
\begin{align*}
| w_i - w| = |\phi_{k,i}^2 - \phi_k^2| = |( \phi_{k,i} + \phi_k) ( \phi_{k,i} - \phi_k )| \leq C |\phi_{k,i} - \phi_k|.
\end{align*}
Since $\phi_{k,i} \rightarrow \phi_k$ in $L^{\frac{n}{n-2}}(M)$, it follows that $w_i \to w$ in $L^{N/2}(M)$.  Applying Claim \ref{weakClaim}, we get
\begin{align*}
1 = \lim_{i \to \infty} \int_M \phi_{k,i}^2 u_i^{N-2} \, dv_g = \int_M \phi_k^2 \bar{u}^{N-2} \, dv_g,
\end{align*}
which proves (\ref{L21}).  This completes the proof of Proposition \ref{Cprop}.  
\end{proof}

As we indicated above, the next compactness result shows that the eigenvalue bound (\ref{EVB}) can be dropped if assume the weak limit $\bar{u} \in L^N_{>0}(M)$:

\begin{proposition} \label{Cpropplus}  Assume $\nu([g]) \geq 2$, and $\ker L_g$ is trivial.

Let $\{ u_i \} \subset L^N_{>0}$ be a sequence of generalized conformal factors, normalized as in (\ref{normi}), and $\bar{u}^{N-2}$ a weak subsequential limit of $\{ u_i^{N-2} \}$ as in the preamble to Proposition \ref{Cprop}.  For $k=1$ or $2$, let $\phi_{k,i}$ be a generalized $k^{th}$ eigenfunction associated to $u_i$, normalized so that
\begin{align} \label{normii}
\int_M \phi_{k,i}^2 u_i^{N-2} \, dv_g = 1.
\end{align}

If $\bar{u} \in L^N_{>0}(M)$, then there is a constant $C > 0$ such that
\begin{align} \label{WWa}
\| \phi_{k,i} \|_{W^{1,2}(M)} + \| \phi_{k,i} \|_{L^{\infty}(M)} \leq C.
\end{align}
Moreover, a subsequence of $\{ \phi_{k,i} \}$ converges weakly in $W^{1,2}(M)$, and strongly in $L^{N/2}(M)$, to
$\phi_k \in W^{1,2}(M) \cap L^{\infty}(M)$, a (non-trivial) $k^{th}$-generalized eigenfunction associated to $\bar{u}$, and satisfying the normalization
\begin{align} \label{2112}
\int_M \phi_k^2 \bar{u}^{N-2} \, dv_g = 1.
\end{align}
\end{proposition}

\begin{proof}  We first show that the assumption $\bar{u} \in L^N_{>0}(M)$ prevents $\lambda_1(u_i)$ from diverging as $i \to \infty$:

\begin{lemma} \label{LN1Lemma} There is a constant $C > 0$ such that
\begin{align} \label{LN1}
\lambda_2(u_i) > \lambda_1(u_i) \geq -C
\end{align}
for all $i$.
\end{lemma}

\begin{proof}  Assume on the contrary that $\lambda_1(u_i)\rightarrow -\infty$ along some subsequence. As in the proof of Proposition \ref{Cprop}, let
\begin{align} \label{scor1}
\tilde{\phi}_{1,i} = \dfrac{ \phi_{1,i} }{ \| \phi_{1,i} \|_{L^2(M)}}.
\end{align}
Then $\tilde{\phi}_{k,i}$ is a generalized first eigenfunction with
\begin{align} \label{newnormie}
\| \tilde{\phi}_{1,i} \|_{L^2(M)} = 1,
\end{align}
and is a $W^{1,2}(M)$-solution of
\begin{align}
-\Delta_g \tilde{\phi}_{1,i} + c_n R_g \tilde{\phi}_{1,i} = \lambda_1(u_i) \tilde{\phi}_{1,i} u_{i}^{N-2}.
\end{align}
Using $\tilde{\phi}_{1,i}$ as a test function we have
\begin{align} \label{testyy2}
\int_M \left( |\nabla_g \tilde{\phi}_{1,i}|^2 + c_n R_g \tilde{\phi}_{1,i}^2 \right) \, dv_g = \lambda_1(u_i) \int_M \tilde{\phi}_{1,i}^2 u_{\epsilon}^{N-2} \, dv_g.
\end{align}
From this inequality we conclude two things.  First, since $\lambda_1(u_i) < 0$, we have
\begin{align*}
\int_M |\nabla_g \tilde{\phi}_{1,i}|^2 \, dv_g \leq -c_n \int_M R_g \tilde{\phi}_{1,i}^2  \, dv_g  \leq C,
\end{align*}
hence
\begin{align} \label{W12t2}
\| \tilde{\phi}_{1,i} \|_{W^{1,2}(M)} \leq C.
\end{align}
Therefore, there is a $\tilde{\phi}_1 \in W^{1,2}(M)$ such that a subsequence of $\{ \tilde{\phi}_{1,i} \}$ (still denoted $\{ \tilde{\phi}_{1,i} \}$) converges weakly in $W^{1,2}(M)$, and strongly in $L^{N/2}(M)$, to $\tilde{\phi}_1$.  By strong convergence,
\begin{align} \label{SSP1}
\int_M \tilde{\phi}_1^2 \, dv_g = 1.
\end{align}

The second consequence of (\ref{testyy2}) is the following: since the left-hand side of (\ref{testyy2}) is bounded below by a constant and $\lambda_1(u_i) \rightarrow -\infty$, we must have
\begin{align} \label{slam}
\int_M \tilde{\phi}_{1,i}^2 u_i^{N-2} \, dv_g \rightarrow 0, \ \ i \to \infty.
\end{align}
Therefore,
\begin{align*}
\int_M \tilde{\phi}_{1,i} u_i^{N-2} \, dv_g &\leq \left( \int_M \tilde{\phi}_{1,i}^2 u_i^{N-2} \, dv_g \right)^{1/2} \left( \int_M u_i^{N-2} \, dv_g \right)^{1/2} \\
& \rightarrow 0, \ \ i \to \infty.
\end{align*}
Since $\{ \tilde{\phi}_{1,i} \}$ converges strongly in $L^{N/2}(M)$ to $\tilde{\phi}_1$, we conclude
\begin{align} \label{flail}
\int_M \tilde{\phi}_1 \bar{u}^{N-2} \, dv_g = \lim_{i \to \infty} \int_M \tilde{\phi}_{1,i} u_i^{N-2} \, dv_g = 0.
\end{align}
This implies $\tilde{\phi}_1 \bar{u}^{N-2} = 0$.  Since $\bar{u} \in L^N_{>0}(M)$, it follows that $\tilde{\phi}_1 = 0$ a.e., which contradicts  (\ref{SSP1}).  The lemma follows.
\end{proof}

Since Lemma \ref{LN1Lemma} gives us uniform bounds for $|\lambda_k(u_i)|$, by Proposition \ref{Cprop} we have uniform bounds for the corresponding generalized eigenfunctions:
\begin{align} \label{WWab}
\| \phi_{k,i} \|_{W^{1,2}(M)} + \| \phi_{k,i} \|_{L^{\infty}(M)} \leq C.
\end{align}
Therefore, a subsequence of $\{ \phi_{k,i} \}$ converges weakly in $W^{1,2}(M)$, and strongly in $L^{N/2}(M)$, to $\phi_k \in W^{1,2}(M) \cap L^{\infty}(M)$ satisfying
\begin{align} \label{geve1}
L_g \phi_k = \bar{\lambda}_k \phi_k \bar{u}^{N-2},
\end{align}
where $\lambda_k(u_i) \rightarrow \bar{\lambda}_k$.  Also, $\phi_k$ satisfies
\begin{align} \label{L211}
\int_M \phi_k^2 \bar{u}^{N-2} \, dv_g = 1.
\end{align}
It remains to show that for $k=1$ and $k=2$, $\bar{\lambda}_k = \lambda_k(\bar{u})$, hence $\phi_k$ is a $k^{th}$-generalized eigenfunction associated to $\bar{u}$.  We remark that $\lambda_k(\bar{u})$ is well defined since we are assuming $\bar{u} \in L^N_{>0}$.

We first show that $\bar{\lambda}_1 = \lambda_1(\bar{u})$.  Let $\bar{\phi}_1$ be a generalized first eigenfunction associated to $\bar{u}$.  We may assume $\bar{\phi}_1 \geq 0$.  Since $\bar{\phi}_1$ is a $W^{1,2}$-solution of
\begin{align} \label{ww4}
L_g \bar{\phi}_1 = \lambda_1(\bar{u}) \bar{\phi}_1 \bar{u}^{N-2},
\end{align}
we can use $\phi_1$ as a test function in (\ref{ww4}) and $\bar{\phi}_1$ as a test function in (\ref{ww3}) to conclude
\begin{align*}
 \lambda_1(\bar{u}) \int_M \phi_1 \bar{\phi}_1 \bar{u}^{N-2} \, dv_g = \bar{\lambda}_1 \int_M \bar{\phi}_1 \phi_1 \bar{u}^{N-2} \, dv_g,
 \end{align*}
 hence
 \begin{align*}
 ( \lambda_1(\bar{u}) - \bar{\lambda}_1 ) \int_M \phi_1 \bar{\phi}_1 \bar{u}^{N-2} \, dv_g = 0.
 \end{align*}
If $\lambda_1(\bar{u}) \neq \bar{\lambda}_1$, then since $\phi_1 \bar{\phi}_1 \geq 0$ and $\bar{u} \in L^N_{>0}$, it follows that $\phi_1 \bar{\phi}_1 = 0$ a.e.  However, by part (i) of Proposition \ref{efProp}, $\bar{\phi}$ can only vanish on a set of zero measure.  Since $\phi_1 \bar{\phi}_1 = 0$ a.e., it must follow that $\phi_1 = 0$ a.e., which contradicts (\ref{L211}).  Therefore, $ \lambda_1(\bar{u}) = \bar{\lambda}_1$.

To see that $\bar{\lambda}_2 = \lambda_2(\bar{u})$, we use the characterization of $\lambda_2$ from part $(iii)$ of Proposition \ref{efProp}. We begin with a claim:

\begin{claim} \label{ConClaim}
\begin{align} \label{bally}
\int_M \phi_1 \phi_2 \bar{u}^{N-2} \, dv_g = 0,
\end{align}
and
\begin{align} \label{ballyhoo}
\lim_{i \to \infty} \int_M \phi_2 \phi_{1,i} u_i^{N-2} \, dv_g = 0.
\end{align}
\end{claim}

\begin{proof}  To prove (\ref{bally}), we use the fact that
\begin{align} \label{bally2}
\int_M \phi_{1,i} \phi_{2,i} u_i^{N-2} \, dv_g = 0.
\end{align}
Recall $\phi_{k,i}$ converges to $\phi_k$ in $L^{N/2}(M)$ and is moreover bounded.  It follows that $\phi_{1,i} \phi_{2,i}$ converges to $\phi_1 \phi_2$ in $L^{N/2}(M)$, since
\begin{align*}
| \phi_{1,i} \phi_{2,i} - \phi_1 \phi_2 | &\le | (\phi_{1,i} - \phi_1)\phi_{2,i}| + |(\phi_{2,i} - \phi_2)\phi_1| \\
&\leq C  |\phi_{1,i} - \phi_1| + |\phi_{2,i} - \phi_2|.
\end{align*}
Therefore, by Claim \ref{weakClaim}
\begin{align*}
0 = \lim_{i \to \infty} \int_M  \phi_{1,i} \phi_{2,i} u_i^{N-2} \, dv_g = \int_M \phi_1 \phi_2 \bar{u}^{N-2} \, dv_g,
\end{align*}
which proves (\ref{bally}).

The proof of (\ref{ballyhoo}) is similar.  Again, since $\phi_{1,i}$ converges to $\phi_1$ in $L^{N/2}(M)$ and $\phi_2$ is bounded, it follows that $\phi_2 \phi_{1,i}$ to $\phi_2 \phi_1$ in $L^{N/2}(M)$, hence
\begin{align*}
\lim_{i \to \infty} \int_M \phi_2 \phi_{1,i} u_i^{N-2} \, dv_g = \int_M \phi_2 \phi_1 \bar{u}^{N-2} \, dv_g = 0.
\end{align*}
\end{proof}

As a consequence of (\ref{bally}) and part $(iii)$ of Proposition \ref{efProp},
\begin{align} \label{Ray2}
\lambda_2(\bar{u}) \leq \dfrac{ \int_M \left( |\nabla_g \phi_2 |^2 + c_n R_g \phi_2^2 \right) \, dv_g }{ \int \phi_2^2 u^{N-2} \, dv_g }.
\end{align}
By (\ref{L21}) and the fact that $\{ \phi_{2,i} \}$ converges weakly in $W^{1,2}(M)$ to $\phi_2$,
\begin{equation} \label{Ray3}
\begin{split}
\dfrac{ \int_M \left( |\nabla_g \phi_2 |^2 + c_n R_g \phi_2^2 \right) \, dv_g }{ \int \phi_2^2 u^{N-2} \, dv_g } &\leq \liminf_{i \to \infty} \dfrac{ \int_M \left( |\nabla_g \phi_{2,i}|^2 + c_n R_g \phi_{2,i}^2 \right) \, dv_g }{ \int_M \phi_{2,i}^2  u_i^{N-2} \, dv_g }\\ &= \liminf_{i \to \infty} \lambda_{2}(u_i) = \bar{\lambda}_2.
\end{split}
\end{equation}
By (\ref{Ray2}) we conclude
\begin{align} \label{Ray4}
\lambda_2(\bar{u}) \leq \bar{\lambda}_2.
\end{align}

To prove the opposite inequality, we argue as follows:  For each $i$ let
\begin{align} \label{ci1}
\tilde{\phi}_{2,i} = \phi_2 - \alpha_i \phi_{1,i},
\end{align}
where $\alpha_i$ is given by
\begin{align} \label{ci2}
\alpha_i = \int_M \phi_2 \phi_{1,i} u_i^{N-2} \, dv_g.
\end{align}
Then
\begin{align} \label{perp1}
\int_M \tilde{\phi}_{2,i} \phi_{1,i} u_i^{N-2} \, dv_g = 0.
\end{align}
Therefore, by part $(iii)$ of Proposition \ref{efProp},
\begin{align} \label{Rayz1}
\lambda_2(u_i) \leq \dfrac{ \int_M \left( |\nabla_g \tilde{\phi}_{2,i} |^2 + c_n R_g \tilde{\phi}_{2,i}^2 \right)\, dv_g }{ \int_M \tilde{\phi}_{2,i}^2 u_i^{N-2} \, dv_g }.
\end{align}

By (\ref{ballyhoo}), $\alpha_i \rightarrow 0$ as $i \to \infty$. Therefore,
\begin{align} \label{elim}
\int_M \left( |\nabla_g \tilde{\phi}_{2,i} |^2 + c_n R_g \tilde{\phi}_{2,i}^2 \right)\, dv_g \rightarrow  \int_M \left( |\nabla_g \phi_2 |^2 + c_n R_g \phi_2^2 \right)\, dv_g, \ \ i \to \infty.
\end{align}
Also, by weak convergence of $u_i^{N-2}$,
\begin{align*}
\int_M \phi_2^2 u_i^{N-2} \, dv_g \rightarrow \int_M \phi_2^2 \bar{u}^{N-2} \, dv_g = 1,
\end{align*}
which implies
\begin{align} \label{Dlim}
\int_M \tilde{\phi}_{2,i}^2 u_i^{N-2} \, dv_g \rightarrow 1, \ \ i \to \infty.
\end{align}
By (\ref{Rayz1}), (\ref{elim}), and (\ref{Dlim}), we obtain
\begin{align*}
\lim_{i \to \infty} \lambda_2(u_i) &\leq \dfrac{\int_M \left( |\nabla_g \phi_2 |^2 + c_n R_g \phi_2^2 \right)\, dv_g }{ \int_M \phi_2^2 \bar{u}^{N-2} \, dv_g }  \\
&= \lambda_2(\bar{u}).
\end{align*}
Combining this with (\ref{Ray4}), we have $\bar{\lambda}_2 = \lambda_2(\bar{u})$.
\end{proof}


In Section \ref{VariationFormulas} it will sometimes be convenient to choose a different normalization for the sequence of generalized eigenfunctions.  We therefore record the following corollary:

\begin{corollary} \label{Ccor} Same assumptions as Proposition \ref{Cpropplus}.  Let
\begin{align} \label{psiki}
\psi_{k,i} = \beta_{k,i} \phi_{k,i},
\end{align}
where $\beta_{k,i}$ is chosen so that $\psi_{k,i}$ satisfies the normalization
\begin{align} \label{mixnorm}
\int_M \psi_{k,i}^2 \bar{u}^{N-2} \, dv_g = 1.
\end{align}
Then there is a $C = C(g)$ such that
\begin{align} \label{WWpsi}
\| \psi_{k,i} \|_{W^{1,2}(M)} + \| \psi_{k,i} \|_{L^{\infty}(M)} \leq C.
\end{align}
Also, after possibly restricting to a subsequence, $\{ \psi_{k,i} \}$ converges weakly in $W^{1,2}(M)$, and strongly in $L^{N/2}(M)$, to $\phi_k \in W^{1,2}(M) \cap L^{\infty}(M)$, a (non-trivial) $k^{th}$-generalized eigenfunction associated to $\bar{u}$ satisfying the normalization (\ref{L211}).
\end{corollary}

\begin{proof}  By (\ref{mixnorm}),
\begin{align*}
\beta_{k,i}^2 = \left( \int_M \phi_{k,i}^2 \bar{u}^{N-2} \, dv_g \right)^{-1}.
\end{align*}
It follows from the estimates of Proposition \ref{Cpropplus} that along a subsequence,
\begin{align*}
\int_M \phi_{k,i}^2 \bar{u}^{N-2} \, dv_g \rightarrow \int_M \phi_k^2 \bar{u}^{N-2} \, dv_g = 1, \ \ i \to \infty.
\end{align*}
Therefore, $\beta_{k,i} \rightarrow 1$ as $i \to \infty$, and the estimates in (\ref{WWpsi}) follow from the estimates (\ref{WW}).  Also, it is clear that $\psi_{k,i}$ converges weakly in $W^{1,2}(M)$, and strongly in $L^{N/2}(M)$, to $\phi_k \in W^{1,2}(M) \cap L^{\infty}(M)$, the limit in the conclusion of Proposition \ref{Cprop}.
\end{proof}

\section{First Variation Formulas}\label{VariationFormulas}

Throughout this section we assume that we are given a function $u\in L^{N}_{>0}$.  By Propositions \ref{Prop22} and \ref{efProp}, this guarantees $\lambda_1(u)>-\infty$ and the existence of generalized eigenfunctions.  For any $h \in L^\infty := L^\infty(M^n,g)$, let
\begin{equation}
u_t = u + th u = u(1+th).
\end{equation}
We refer to $h$ as the generating function of the deformation $u_t$ of $u$, and it will be fixed throughout this section. The goal is to understand the behavior of $\lambda_k(u_t)$ around $t=0$ for $k=1$ and $k=2$, thus it is always possible to think of $t$ as living in a small neighborhood $(-\delta,\delta)$ of zero. In particular, we choose $\delta>0$ small enough such that
\begin{equation}
|1+th| \ge 1 - \delta \|h\|_{\infty} > 0.
\end{equation}

\begin{proposition}\label{T-functions} Test functions for $\lambda_k(u)$ ($k=1$ or $2$) are also test functions for $\lambda_k(u_t)$, and vice versa. Also, $\lambda_k(u_t)$ is finite for all $t\in (-\delta,\delta)$.
\end{proposition}

\begin{proof}
The proof follows from Lemma \ref{TestFunctions} after noticing that the positive set of $u$ and $u_t$ are identical for $t\in(-\delta, \delta)$.
\end{proof}
In what follows the continuity of $t\longmapsto\lambda_k(u_t)$ at $t=0$ is proven in the cases where $k=1$ or $k=2$.

\begin{proposition}\label{cont-1} $\lim_{t\to 0} \lambda_1(u_t) = \lambda_1(u)$ and $\lim_{t\to 0} \lambda_2(u_t) = \lambda_2(u)$.
\end{proposition}

\begin{proof}
By Proposition \ref{T-functions}, $Gr_1^u(W^{1,2}) = Gr_1^{u_t}(W^{1,2})$ for all $t\in(-\delta,\delta)$. Since the sign of both $\lambda_1(u_t)$ and $\lambda_1(u)$ is negative, among these test functions it is enough to consider those $\phi \in Gr_1^u(W^{1,2})$ for which
\begin{equation}\label{cont-1-1}
\int_M \phi L_g \phi \;dv_g <0.
\end{equation}
Now, observe that
\begin{equation} \label{norm1}
\begin{split}
\int_M \phi^2 u_t^{N-2}\;dv_g & = \int_M \phi^2u^{N-2}(1+th)^{N-2}\;dv_g\\ &\le (1+|t| \|h\|_\infty)^{N-2} \int_M \phi^2u^{N-2}\;dv_g,
\end{split}
\end{equation}
and
\begin{equation}\label{norm2}
\begin{split}
\int_M \phi^2 u_t^{N-2}\;dv_g & = \int_M \phi^2u^{N-2}(1+th)^{N-2}\;dv_g \\ &\ge (1-|t|\|h\|_\infty)^{N-2}\int_M \phi^2u^{N-2}\;dv_g.
\end{split}
\end{equation}
Therefore,
\begin{equation}\label{cont-1-2}
 (1-|t|\|h\|_\infty)^{-(N-2)}\mathcal{R}_g^u(\phi) \le \mathcal{R}_g^{u_t}(\phi) \le (1+|t|\|h\|_\infty)^{-(N-2)} \mathcal{R}_g^u(\phi)
\end{equation}
for all those $\phi$'s satisfying (\ref{cont-1-1}). After taking the infimum over all such $\phi$'s, we get
\begin{equation}
(1-|t|\|h\|_\infty)^{-(N-2)}\lambda_1(u) \le \lambda_1(u_t)\le (1+|t|\|h\|_\infty)^{-(N-2)} \lambda_1(u),
\end{equation}
and the continuity of $t\mapsto \lambda_1(u_t)$ at $t=0$ follows.

For the continuity of $t\mapsto \lambda_2(u_t)$, notice that by Proposition \ref{NegativeEigen}, both $\lambda_2(u_t)$ and $\lambda_2(u)$ are negative. Therefore, Definition \ref{GeneralizedEigenvalue} implies that we can take two dimensional subspaces $\Sigma\in Gr^u_2(W^{1,2}) = Gr^{u_t}_2 (W^{1,2})$ satisfying
\begin{equation}
\int_M \phi L_g \phi\;dv_g < 0
\end{equation}
for all $\phi \in \Sigma$. Hence, the result for $t\mapsto\lambda_2(u_t)$ follows in a similar manner from (\ref{cont-1-2}).
\end{proof}

 Let $E_1(u_t)$ denote the space of generalized first eigenfunctions associated to $u_t$.  Since $E_1(u_t)$ can be considered a subspace in either $L^2(u_t) := L^2(M^n,u_t^{N-2}dv_g)$ or $L^2(u) := L^2(M^n,u^{N-2}dv_g)$, there are at least two possible orthogonal projections that we could define. Let us denote by $P_t^*$ and $P_t$ the orthogonal projections onto $E_1(u_t)$ in $L^2(u_t)$ and $L^2(u)$, respectively. We observe the following:

\begin{lemma}\label{Projections}
Let $P_t^*$ and $P_0$ be as in the preceding paragraph. For any $\phi\in E_2(u)$ normalized such that $\int_M \phi^2u^{N-2}\;dv_g = 1$,
\begin{equation} \label{Pr-1}
\|P_t^*\phi\|_{L^2(u_t)}=O(t^2),
\end{equation}
as $t\to 0$. Also, for any $\phi_t\in E_2(u_t)$ with $\int_M \phi_t^2u^{N-2}\;dv_g=1$, we have
\begin{equation}\label{Pr-2}
\|P_0\phi_t \|_{L^2(u)} = O(t^2),
\end{equation}
as $t\to 0$.
\end{lemma}

\begin{proof}
Let $\psi_{1,t}$ be the spanning eigenfunction of $E_1(u_t)$ normalized such that
\begin{equation}
\int_M \psi_{1,t}^2u_t^{N-2}\;dv_g = 1.
\end{equation}
As $E_1(u_t)$ is one dimensional, for any $\phi \in E_2(u)$ we deduce
\begin{equation}
\begin{split}
\|P^*_t\phi\|^2_{L^2(u_t)} & = |\langle \phi, \psi_{1,t} \rangle_{L^2(u_t)}|^2\\
& = \left|\int_M \phi \psi_{1,t}u_t^{N-2}\;dv_g\right|^2
\end{split}
\end{equation}
Since $\psi_{1,t} \in E_1(u_t)$ and $\phi \in E_2(u)$, the following two equations are satisfied:
\begin{equation} \label{weq}
L_g\phi = \lambda_2(u)\phi u^{N-2}
\end{equation}
and
\begin{equation}
L_g \psi_{1,t} = \lambda_1(u_t) \psi_{1,t} u_t^{N-2}.
\end{equation}
Testing the above equation against $\phi$ gives us
\begin{equation}
\begin{split}
\lambda_1(u_t) \int_M \phi \psi_{1,t} u_t^{N-2}\;dv_g & = \lambda_2(u) \int_M \phi \psi_{1,t} u^{N-2}\;dv_g \\ & = \lambda_2(u) \int_M \phi \psi_{1,t} (u^{N-2} - u_t^{N-2})\;dv_g\\ &\hspace{.15in}+ \lambda_2(u) \int_M \phi \psi_{1,t} u_t^{N-2}\;dv_g,
\end{split}
\end{equation}
where we have used (\ref{weq}) to get the first equality.
Rearranging the terms above we deduce
\begin{equation} \label{Pr-1-eq1}
(\lambda_2(u) - \lambda_1(u_t)) \int_M \phi \psi_{1,t} u_t^{N-2}\;dv_g =  \lambda_2(u) \int_M \phi \psi_{1,t} (u_t^{N-2} - u^{N-2})\;dv_g
\end{equation}
Recall that we are working in a neighborhood of $t=0$ on which $|th|<1$. Then, by the power series expansion of $(1+th)^{N-2}$,
\begin{equation}\label{Pr-1-eq2}
\begin{split}
u_t^{N-2} - u^{N-2} &= u^{N-2}[(1+th)^{N-2} - 1] \\ & = (N-2)th\cdot u^{N-2} + O(t^2)\cdot u^{N-2}
\end{split}
\end{equation}
Also, by Proposition \ref{cont-1},
\begin{equation}
\lambda_1(u_t) = \lambda_1(u) + o(1),
\end{equation}
and since $|\lambda_2(u) - \lambda_1(u)|\ge \gamma > 0$ by part $(iv)$ of Proposition \ref{efProp}, this implies
\begin{equation}\label{Pr-1-eq3}
|\lambda_2(u) - \lambda_1(u_t)| \ge \frac{\gamma}{2}
\end{equation}
for small $t$. Combining (\ref{Pr-1-eq1}), (\ref{Pr-1-eq2}), and (\ref{Pr-1-eq3}), we conclude
\begin{equation}
\begin{split}
\frac{\gamma^2}{4} \left|\int_M \phi \psi_{1,t} u_t^{N-2} \;dv_g\right|^2 &\le Ct^2 \left(\int_M|\phi| \psi_{1,t} u^{N-2} \; dv_g\right)^2 \\ &\le Ct^2.
\end{split}
\end{equation}
This finishes the proof of (\ref{Pr-1}).

To prove (\ref{Pr-2}), note that
\begin{equation}
\|P_0\phi_t\|_{L^2(u)}^2 = \left|\int_M\phi_t\psi_{1,0}u^{N-2}\;dv_g\right|^2
\end{equation}
Using the equation that $\phi_t$ and $\psi_{1,0}$ satisfy weakly, we deduce
\begin{equation}
\begin{split}
\lambda_1(u)\int_M \phi_t\psi_{1,0}u^{N-2}\;dv_g &= \lambda_2(u_t)\int_M \psi_{1,0}\phi_tu_t^{N-2}\;dv_g\\ &= \lambda_2(u_t)\int_M \psi_{1,0}\phi_t(u_t^{N-2} - u^{N-2})\;dv_g\\ &\hspace{.15in} + \lambda_2(u_t)\int_M \psi_{1,0}\phi_tu^{N-2}\;dv_g
\end{split}
\end{equation}
The remaining part of the argument is similar to the one used for (\ref{Pr-1}) and is omitted.
\end{proof}

\begin{proposition}\label{Ines} The following two inequalities hold:
\begin{equation}\label{Ine1}
\lambda_2(u_t)\le \inf_{\phi \in E_2(u)} \mathcal{R}_g^{u_t}(\phi) + o(t)\;\;\;\; (t\to 0),
\end{equation}
and
\begin{equation}\label{Ine2}
\lambda_2(u)\le \inf_{\phi \in E_2(u_t)} \mathcal{R}_g^u(\phi) + o(t)\;\;\;\; (t\to 0).
\end{equation}
\end{proposition}

\begin{proof}
 Let $\phi \in E_2(u)$ be arbitrary, and assume that $\int_M \phi ^2u^{N-2}\;dv_g=1$. We estimate the difference between $\mathcal{R}^{u_t}_g(\phi-P_t^*\phi)$ and $\mathcal{R}_g^{u_t}(\phi)$ as follows:
\[
\begin{split}
&\left|\frac{\int_M \{ |\nabla_g(\phi  - P_t^*\phi)|^2 + c_nR_g (\phi-P_t^*\phi)^2\}\;dv_g}{\int_M (\phi-P_t^*\phi)^2u_t^{N-2}\;dv_g}  - \frac{\int_M \{ |\nabla_g \phi|^2 + c_nR_g \phi^2\}\;dv_g}{\int_M \phi^2u_t^{N-2}\;dv_g}\right| \\
& = \left|\frac{\lambda_2(u) - \lambda_1(u_t)  \int_M (P_t^*\phi)^2u_t^{N-2} \;dv_g}{ \int_M (\phi-P_t^*\phi)^2u_t^{N-2} \;dv_g} - \frac{\lambda_2(u)}{ \int_M \phi^2u_t^{N-2}\;dv_g}\right| \\
& = \Bigg{|}\frac{\lambda_2(u) \left\{\int_M \phi^2u_t^{N-2}\;dv_g - \int_M (\phi-P_t^*\phi)^2u_t^{N-2}\;dv_g\right\}}{ \int_M (\phi -P_t^*\phi)^2u_t^{N-2} \; dv_g \cdot \int_M \phi^2u_t^{N-2}\;dv_g} \\
& \hspace{.15in} -\frac{ \lambda_1(u_t) \int_M (P_t^*\phi)^2u_t^{N-2}\;dv_g\cdot \int_M \phi^2u_t^{N-2}\;dv_g}{ \int_M (\phi -P_t^*\phi)^2u_t^{N-2} \; dv_g \cdot \int_M \phi^2u_t^{N-2}\;dv_g}\Bigg{|} \\
& =\left|\frac{\lambda_2(u)\int_M (P_t^*\phi)^2u_t^{N-2}\;dv_g -\lambda_1(u_t) \int_M (P_t^*\phi)^2u_t^{N-2}\;dv_g\cdot \int_M \phi^2u_t^{N-2}\;dv_g}{\int_M (\phi-P_t^*\phi)^2u_t^{N-2} \; dv_g \cdot \int_M \phi^2u_t^{N-2}\;dv_g}\right| \\
& = \left| \frac{ \int_M (P_t^*\phi)^2u_t^{N-2}\;dv_g}{\int_M (\phi-P_t^*\phi)^2u_t^{N-2}\;dv_g} \cdot \frac{\lambda_2(u) - \lambda_1(u_t) \int_M \phi^2 u_t^{N-2}\;dv_g }{\int_M \phi^2u_t^{N-2}\;dv_g}\right| \\
& \le \frac{ \int_M (P_t^*\phi)^2u_t^{N-2}\;dv_g}{\int_M (\phi-P_t^*\phi)^2u_t^{N-2}\;dv_g}\cdot \left\{\frac{|\lambda_2(u)|}{\int_M \phi^2 u_t^{N-2}\;dv_g} + |\lambda_1(u_t)|\right\}
\end{split}
\]
It follows from Lemma \ref{Projections} that this last expression is of order $O(t^2)$ as $t\to 0$. From the above estimates we deduce,
\begin{equation}
\lambda_2(u_t)\le \frac{\int_M \{ |\nabla_g(\phi  - P_t^*\phi)|^2 + c_nR_g (\phi-P_t^*\phi)^2\}\;dv_g}{\int_M (\phi-P_t^*\phi)^2u_t^{N-2}\;dv_g} \le \mathcal{R}_g^{u_t}(\phi) + O(t^2).
\end{equation}
Therefore, taking the infimum over all $\phi\in E_2(u)$,
\begin{equation}
\lambda_2(u_t) \le \inf_{\phi\in E_2(u)} \mathcal{R}_g^{u_t}(\phi) + O(t^2)\;\;\;\; (|t|\to0).
\end{equation}

As for inequality (\ref{Ine2}), we select an arbitrary function $\phi \in E_2(u_t)$, and estimate the difference between $\mathcal{R}_g^{u}(\phi -P_0\phi)$ and $\mathcal{R}_g^u(\phi)$. Since the argument is similar we omit the details of the proof.
\end{proof}

For any generating function $h\in L^\infty $, we define $L_h(\cdot, u)$ for functions in $L^2(u)$ or $L^2(u_t)$ by
\begin{equation}
L_h(\phi,u):= -(N-2)\mathcal{R}_g^u(\phi)\cdot \frac{ \int_M h\phi^2 u^{N-2}\;dv_g}{\int_M \phi^2u^{N-2}\;dv_g}
\end{equation}
For functions $\phi \in E_2(u)$ with $\int_M \phi^2u^{N-2}\;dv_g=1$, $L_h(\phi,u)$ takes the simpler form
\begin{equation}-(N-2)\lambda_2(u) \int_M h\phi^2 u^{N-2}\;dv_g.\end{equation}
Note that if one think of $\mathcal{R}_g^{u_t}(\phi)$ as a function in $t\in (-\delta,\delta)$, where $\phi \in L^2(u)$ is fixed, the functional $L_h(\phi,u)$ is merely the derivative of $\mathcal{R}_g^{u_t}(\phi)$ at $t=0$.

\begin{proposition} \label{liminf-limsup-L} The following two limits hold:
\begin{equation}\label{Conv-Linf}
\lim_{t\to0} (\inf_{\phi\in E_2(u_t)} L_h(\phi,u))  = \inf_{\phi\in E_2(u)} L_h (\phi,u),
\end{equation}
and
\begin{equation}\label{Conv-Lsup}
\lim_{t\to0} (\sup_{\phi\in E_2(u_t)} L_h(\phi,u))  = \sup_{\phi\in E_2(u)} L_h (\phi,u).
\end{equation}
\end{proposition}

\begin{proof}
Let $\Pi_t$ be the orthogonal projection from $L^2(u)$ onto $E_2(u_t)$. Our first goal is to show that
\begin{align*}
A_t(\phi)& :=|L_h(\phi,u)-L_h(\Pi_t\phi,u)| \\ & =(N-2)\left|\mathcal{R}_g^u(\phi)\frac{\int_M h \phi^2 u^{N-2}\;dv_g}{\int_M \phi^2u^{N-2}\;dv_g} - \mathcal{R}_g^u(\Pi_t\phi)\frac{\int_M h (\Pi_t\phi)^2 u^{N-2}\;dv_g}{\int_M (\Pi_t\phi)^2u^{N-2}\;dv_g}\right| \to 0,
\end{align*}
uniformly for $\phi\in E_2(u)$. To see this, we estimate
\begin{align*}
(N-2)^{-1}A_t(\phi)& \le \underbrace{|\lambda_2(u)|\left| \frac{\int_M h\phi^2 u^{N-2}\;dv_g}{\int_M \phi^2u^{N-2}\;dv_g} - \frac{\int_M h(\Pi_t\phi)^2 u^{N-2}\;dv_g}{\int_M (\Pi_t\phi)^2u^{N-2}\;dv_g}\right|}_{:=B_t(\phi)} \\
 &\hspace{.15in} + \|h\|_\infty\underbrace{\left|\lambda_2(u) - \lambda_2(u_t)\frac{\int_M (\Pi_t\phi)^2u_t^{N-2}\;dv_g}{\int_M (\Pi_t\phi)^2u^{N-2}\;dv_g}\right|}_{:=C_t(\phi)}
\end{align*}
The term $B_t(\phi)$ is bounded from above by
\begin{align*}
B_t(\phi) &\le |\lambda_2(u)|\cdot\left|\frac{\left(\int_M h\phi^2 u^{N-2}\;dv_g - \int_M h (\Pi_t\phi)^2 u^{N-2}\;dv_g\right)\cdot \int_M (\Pi_t\phi)^2u^{N-2}\;dv_g}{\int_M \phi^2u^{N-2}\;dv_g\cdot\int_M (\Pi_t\phi)^2u^{N-2}\;dv_g}\right| \\
&\hspace{.15in} + |\lambda_2(u)| \cdot\left|\frac{\left(\int_M(\Pi_t\phi)^2 u^{N-2}\;dv_g - \int_M \phi^2 u^{N-2}\;dv_g\right)\cdot \int_M h (\Pi_t\phi)^2 u^{N-2}\;dv_g}{\int_M \phi^2u^{N-2}\;dv_g\cdot\int_M (\Pi_t\phi)^2u^{N-2}\;dv_g}\right| \\
& \le C\cdot\frac{\int_M |(\Pi_t\phi)^2-\phi^2|u^{N-2}\;dv_g}{ \int_M \phi^2u^{N-2}\;dv_g},
\end{align*}
which goes to zero uniformly in $\phi \in E_2(u)$ by Lemma \ref{Pr-3}. As for $C_t(\phi)$, we have
\begin{equation}
\begin{split}
C_t(\phi) &\le |\lambda_2(u) - \lambda_2(u_t)| + |\lambda_2(u_t)|\cdot \left|1 - \frac{\int_M (\Pi_t \phi)^2 u_t^{N-2}\;dv_g}{ \int_M (\Pi_t\phi)^2u^{N-2}\;dv_g} \right|\\
&= |\lambda_2(u) - \lambda_2(u_t)| + |\lambda_2(u_t)|\cdot \left|\frac{\int_M (\Pi_t \phi)^2(u^{N-2}- u_t^{N-2})\;dv_g}{ \int_M (\Pi_t\phi)^2u^{N-2}\;dv_g} \right|\\
&\le |\lambda_2(u) - \lambda_2(u_t)| + C|\lambda_2(u_t)||t|\cdot (1).
\end{split}
\end{equation}
This last expression goes to zero by Proposition \ref{cont-1}. Therefore, our first step is finished, and we conclude that
\begin{equation}
\lim_{t\to 0} \inf_{\phi \in E_2(u)}L_h(\Pi_t\phi,u) = \inf_{\phi \in E_2(u)} L_h(\phi,u).
\end{equation}

An immediate consequence of the previous equality is
\begin{equation}
\lim_{t\to 0} \inf_{\phi\in E_2(u_t)}L_h(\phi,u) \le \inf_{\phi\in E_2(u)} L_h(\phi,u) .
\end{equation}
To show the opposite inequality we proceed as follows. Take a sequence of eigenfunctions $\phi_i = \phi_{t_i} \in E_2(u_t)$ normalized such that $\int_M \phi_i^2 u^{N-2}\;dv_g=1$, and satisfying
\begin{equation}
\inf_{\phi\in E_2(u_t)}L_h(\phi,u) + |t| \ge L_h(\phi_i,u)
\end{equation}
Since $u_t \rightarrow u$ in $L^{\infty}$, we can appeal to Corollary \ref{Ccor} to conclude that the sequence $\{ \phi_i \}$ is bounded in $W^{1,2}(M)$, and there is a function $\phi_0$ such that $\phi_i\rightharpoonup \phi_0$ in $W^{1,2}$ and $\phi_i\to \phi_0$ in $L^2$.  Moreover, $\phi_0 \in E_2(u)$ and satisfies $\int_M \phi_0^2u^{N-2}\;dv_g=1$. Therefore,
\begin{equation}
\lim_{t\to 0} \inf_{\phi\in E_2(u_t)}L_h(\phi,u) \ge L_h(\phi_0,u) \ge \inf_{\phi\in E_2(u)}L_h(\phi,u).
\end{equation}
This concludes the proof of (\ref{Conv-Linf}). The proof of (\ref{Conv-Lsup}) is similar and hence is omitted.
\end{proof}

\begin{lemma}\label{Pr-3}
Let $\Pi_t$ be defined as in the proof of Proposition \ref{liminf-limsup-L}. Then $\|\Pi_t\phi -\phi\|_{L^2(u)} \to 0$ for any $\phi\in E_2(u)$ with $\int_M \phi^2u^{N-2}\;dv_g=1$.
\end{lemma}

\begin{proof}
Let $t_k \rightarrow 0$ and let $\{\psi^i_{2,t_k}\}$ be a basis of $E_2(u_{t_k})$, normalized so that
\begin{align} \label{orthoptk}
\int_M\psi^i_{2,t_k}\psi^j_{2,t_k} u_{t_j}^{N-2}\;dv_g = \delta_{ij}.
\end{align}
Since $\lambda_1(u_{t_k}) \geq -C$, by Proposition \ref{Cprop} for each $i$ the sequence $\{ \psi^i_{2,t_k} \}$ is bounded in $W^{1,2}(M)$, and converges weakly in $W^{1,2}(M)$, and strongly in $L^{N/2}(M)$, to a generalized eigenfunction $\psi^i_{2,0} \in E_2(u)$.  Moreover,
\begin{align} \label{orthoptk2}
\int_M\psi^i_{2,0}\psi^j_{2,0}u^{N-2}\;dv_g = \delta_{ij}.
\end{align}
It follows that $\{ \psi^i_{2,0} \}$ is a basis for $E_2(u)$, hence
 \begin{equation}
 \Pi_{t_k} \phi - \phi = a_i(\phi,t_k)\psi^i_{2,t_k} - a_j(\phi,0)\psi^j_{2,0},
 \end{equation}
 where
\begin{equation}
a_i(\phi,t_k):= \int_M \phi\psi^i_{2,t_k}u^{N-2}\;dv_g.
\end{equation}
Therefore,
\begin{equation} \label{PPt}
\begin{split}
\int_M |\Pi_{t_k} \phi - \phi|^2 u^{N-2}\;dv_g & = \sum_i a_i(\phi,t_k)^2 +\sum_i a_i(\phi,0)^2\\ &\hspace{.15in} - 2a_i(\phi,t_k)a_j(\phi,0)\int_M\psi^i_{2,t_k}\psi^j_{2,0}u^{N-2}\;dv_g,
\end{split}
\end{equation}
where the orthogonality condition (\ref{orthoptk}) has been used. Using the convergence of $\{ \psi^i_{2,t_k} \}$, it is easy to see that
\begin{align*}
a_i(\phi,t_k) &\rightarrow a_i(\phi, 0), \\
\int_M \psi^i_{2,t_k} \psi^j_{2,0} u^{N-2}\;dv_g &\rightarrow \int_M\psi^i_{2,0}\psi^j_{2,0}u^{N-2}\;dv_g = \delta_{ij}
\end{align*}
as $k \to \infty$.  Therefore, the right-hand side of (\ref{PPt}) converges to zero as $k \to \infty$.
\end{proof}

\begin{proposition}\label{OSD} The one-sided derivatives of $\lambda_2(u_t)$ exists at $t=0$. Moreover,
\begin{equation}\label{OSD+}
\frac{d}{dt} \lambda_2(u_t)\Big{|}_{t=0^+} = \inf_{\phi\in E_2(u)} L_h(\phi,u),
\end{equation}
and
\begin{equation}\label{OSD-}
\frac{d}{dt} \lambda_2(u_t)\Big{|}_{t=0^-} = \sup_{\phi\in E_2(u)} L_h(\phi,u)
\end{equation}
\end{proposition}

\begin{proof}
For any smooth function $\phi$ for which $\int_M \phi L_g\phi\;dv_g<0$, we have
\[
\begin{split}
&\left|\frac{1}{t}\left(\int_M \phi^2u_t^{N-2}\;dv_g - \int_M \phi^2u^{N-2}\;dv_g\right) - (N-2)\int_M \phi^2h u^{N-2}\;dv_g\right| \\
=&\left|\frac{1}{t}\int_M \phi^2u^{N-2}\left\{(1+th)^{N-2} - \left(1 +(N-2)th\right)\right\}\; dv_g\right| \\
\le &\int_M \phi^2u^{N-2}\;dv_g\cdot C|t|,
\end{split}
\]
where $C>0$ is a constant depending on $h$ alone. The last inequality in the previous estimates follows from arguments explained in (\ref{Pr-1-eq2}). With this at hand, we compute
\begin{equation}\label{EstimateRL}
\begin{split}
&\left|\frac{1}{t}(\mathcal{R}^{u_t}_g(\phi) - \mathcal{R}_g^u(\phi)) - L_h(\phi,u) \right|\\
&= |\mathcal{R}_g^{u_t}(\phi)|\left| \frac{1}{t}\left( 1 - \frac{\int_M \phi^2u_t^{N-2}\;dv_g}{\int_M \phi^2u^{N-2}\;dv_g}\right) +(N-2) \frac{\int_M \phi^2u_t^{N-2}\;dv_g}{\int_M \phi^2u^{N-2}\;dv_g}\cdot \frac{\int_M h \phi^2 u^{N-2}\;dv_g}{\int_M \phi^2 u^{N-2}\;dv_g}\right|\\
&\le \frac{|\mathcal{R}_g^{u_t}(\phi)|}{ \int_M \phi^2u^{N-2}\;dv_g} \cdot \left|\frac{1}{t}\left(\int_M \phi^2 u_t^{N-2}\;dv_g - \int_M \phi^2u^{N-2}\;dv_g\right) - (N-2)\int_M h \phi^2 u^{N-2}\;dv_g \right| \\
&\hspace{.15in} +  \frac{|\mathcal{R}_g^{u_t}(\phi)|}{ \int_M \phi^2u^{N-2}\;dv_g} \cdot \left| (N-2) \int_M h \phi^2 u^{N-2}\;dv_g\left(1 - \frac{\int_M \phi^2u_t^{N-2}\;dv_g}{ \int_M \phi^2u^{N-2}\;dv_g}\right)\right| \\
&\le |\lambda_1(u_t)| \left\{C|t|+ \tilde C\cdot(N-2)\|h\|_\infty |t| \right\} = o(1),
\end{split}
\end{equation}
where the constant $\tilde C$ is independent of $\phi$.

Now, for $t>0$, and taking the infimum over functions $\phi\in E_2(u)$, we deduce
\begin{equation}
\frac{1}{t} \left(\inf_{\phi\in E_2(u)} \mathcal{R}_g^{u_t}(\phi) - \lambda_2(u)\right) \le \inf_{\phi\in E_2(u)} L_h(\phi,u) + o(1).
\end{equation}
Inequality (\ref{Ine1}) (see Proposition \ref{Ines}) gives us
\begin{equation}
\limsup_{t\to0^+}\frac{\lambda_2(u_t) - \lambda_2(u)}{t} \le \inf_{\phi\in E_2(u)} L_h(\phi,u).
\end{equation}
On the other hand, by plugging into $(\ref{EstimateRL})$ a function $\phi\in E_2(u_t)$ we get
\begin{equation}
L_h(\phi,u) - \frac{1}{t}\left(\lambda_2(u_t) - \mathcal{R}_g^u(\phi)\right)\le o(1).
\end{equation}
Therefore, after taking the infimum over such functions and applying (\ref{Ine2}) with $t>0$,
\begin{equation}
\inf_{\phi\in E_2(u_t)} L_h(\phi,u) \le \frac{1}{t} (\lambda_2(u_t) - \lambda_2(u)) + o(1)
\end{equation}
By means of Proposition \ref{liminf-limsup-L} we finally obtain
\begin{equation}
\inf_{E_2(u)} L_h(\phi,u) \le \liminf_{t\to0^+} \frac{\lambda_2(u_t)-\lambda_2(u)}{t}.
\end{equation}
This finishes the proof for (\ref{OSD+}). The proof of (\ref{OSD-}) is similar and it is, hence, omitted.
\end{proof}

\vskip.2in


\section{Regularized Functional}\label{RegularizedFunctional}


For each $\epsilon>0$, denote by $\mathcal{D}_\epsilon$ the set
\begin{equation}
\mathcal{D}_\epsilon: = \left\{u\in L^N_{>0}:\; u^{-\epsilon}\in L^1\right\},
\end{equation}
where $L^1:= L^1(M,g)$. Note that if $u \in \mathcal{D}_{\epsilon}$, then by Proposition \ref{efProp},$\lambda_2(u)> \lambda_1(u)>-\infty$.  In particular, the functional $F_{2,\epsilon}: \mathcal{D}_\epsilon \rightarrow \mathbb{R}$ given by
\begin{align} \label{Fdef}
F_{2,_\epsilon} (u) &=\lambda_2(u) \left(\int_Mu^N\;dv_g\right)^{\frac{N-2}{N}} - \left( \int_M u^{-\epsilon} \, dv_g \right)  \left(\int_M u^N\;dv_g\right)^{\frac{\epsilon}{N}}
\end{align}
is well defined.  Also, $F_{2,\epsilon}(u) < 0$ for each $u\in \mathcal{D}_\epsilon$.

The following result is a consequence of Proposition (\ref{OSD}):

\begin{proposition} \label{OSDProp} Suppose $u \in  \mathcal{D}_\epsilon$ with
\begin{align} \label{unormal}
\int_M u^{N} \, dv_g = 1.
\end{align}
For $h \in L^{\infty}$, let
\begin{align} \label{ut}
u_t = u(1 + t h).
\end{align}
Then the one-sided derivatives of $F_{2,\epsilon}(u_t)$ at $t=0$ exist, and are given by
\[ \begin{split}
\frac{d}{dt} F_{2,\epsilon}(u_t) \big|_{t = 0^{+}} &= (N-2) \lambda_2(u) \Bigg\{ \Big( 1 - (N-2)^{-1} \frac{\epsilon}{\lambda_2(u)}   \int_M u^{-\epsilon} \, dv_g \Big) \int_M h u^{N} \, dv_g  \\
& \ \ \ \ - \inf_{\phi \in E_2(u)} \dfrac{ \int_M h \phi^2 u^{N-2} \, dv_g}{ \int_M \phi^2 u^{N-2} \, dv_g} + (N-2)^{-1}\frac{\epsilon}{\lambda_2(u)} \int_M h u^{-\epsilon} \, dv_g \Bigg\} , \\
\frac{d}{dt} F_{2,\epsilon}(u_t) \big|_{t = 0^{-}} &= (N-2) \lambda_2(u) \Bigg\{ \Big( 1 - (N-2)^{-1} \frac{\epsilon}{\lambda_2(u)}   \int_M u^{-\epsilon} \, dv_g \Big)  \int_M h u^{N} \, dv_g  \\
& \ \ \ \ - \sup_{\phi \in E_2(u)} \dfrac{ \int_M h \phi^2 u^{N-2} \, dv_g }{ \int_M \phi^2 u^{N-2} \, dv_g}  + (N-2)^{-1}\frac{\epsilon}{\lambda_2(u)} \int_M h u^{-\epsilon} \, dv_g\Bigg\}.
\end{split}
\]
\end{proposition}

\begin{proof} Let
\begin{align} \label{Idef}
\mathcal{I}_{\epsilon}(u) = \left( \int_M u^{-\epsilon} \, dv_g \right)  \left(\int_M u^N\;dv_g\right)^{\frac{\epsilon}{N}},
\end{align}
so that
\begin{align*}
F_{2,\epsilon}(u) =\lambda_2(u) \left(\int_M u^N\;dv_g\right)^{\frac{N-2}{N}} - \mathcal{I}_{\epsilon}(u).
\end{align*}
If $u$ is normalized as in (\ref{unormal}) and $u_t$ is given by (\ref{ut}), then a simple calculation gives
\begin{align} \label{Idot}
\frac{d}{dt} \mathcal{I}_{\epsilon}(u_t) \big|_{t=0} = -\epsilon \int_M h u^{-\epsilon} \, dv_g + \epsilon \left( \int_M u^{-\epsilon} \, dv_g \right) \int_M h u^{N} \, dv_g.
\end{align}
Therefore,
\begin{align*}
\frac{d}{dt} F_{2,\epsilon}(u_t) \big|_{t = 0^{+}} &= \frac{d}{dt}\left( \lambda_2(u_t)\left(\int_M u_t^{N}\;dv_g\right)^{\frac{N-2}{N}}\right)\big|_{t = 0^{+}} - \frac{d}{dt} \mathcal{I}_{\epsilon}(u_t) \big|_{t=0} \\
&= (N-2) \lambda_2(u) \Bigg\{  \int_M h u^{N} \, dv_g - \inf_{\phi\in E_2(u)} \dfrac{ \int_M h \phi^2 u^{N-2} \, dv_g }{ \int_M \phi^2 u^{N-2} \;dv_g}\Bigg\} \\
&\ \ \ \ + \epsilon \int_M h u^{-\epsilon} \, dv_g - \epsilon \left( \int_M u^{-\epsilon} \, dv_g \right) \int_M h u^{N} \, dv_g \\
&= (N-2) \lambda_2(u) \Bigg\{ \Big( 1 - (N-2)^{-1} \frac{\epsilon}{\lambda_2(u)}   \int_M u^{-\epsilon} \, dv_g \Big) \int_M h u^{N} \, dv_g  \\
& \ \ \ \ - \inf_{\phi\in E_2(u)} \dfrac{ \int_M h \phi^2 u^{N-2} \, dv_g }{ \int_M \phi^2 u^{N-2} \, dv_g}  + (N-2)^{-1}\frac{\epsilon}{\lambda_2(u)} \int_M h u^{-\epsilon} \, dv_g\Bigg\},
\end{align*}
as claimed. The computation for the one-sided derivative from the left is analogous.
\end{proof}

\begin{definition}  We say that $u \in  \mathcal{D}_\epsilon$  is {\em extremal} for the regularized functional $F_{2,\epsilon}$ if the one-sided derivatives, $\frac{d}{dt} F_{2,\epsilon}(u_t) \big|_{t = 0^{+}}$ and $\frac{d}{dt} F_{2,\epsilon}(u_t) \big|_{t = 0^{+}}$, have different signs along $u_t:=u(1+th)$ for any generating function $h\in L^\infty$.

An extremal function $u \in  \mathcal{D}_\epsilon$ is said to be {\em maximal}, or a {\em global maximizer}, for $F_{2,\epsilon}$ if
\begin{equation}
F_{2,\epsilon}(u) \ge F_{2,\epsilon}(w)
\end{equation}
for every $w\in \mathcal{D}_{\epsilon}$. Notice that if $u\in \mathcal{D}_\epsilon$ is maximal for $F_{2,\epsilon}$, then
\begin{equation}\label{comp}
\frac{d}{dt} F_{2,\epsilon}(u_t) \big|_{t = 0^{-}}\ge 0 \ge \frac{d}{dt} F_{2,\epsilon}(u_t) \big|_{t = 0^{+}}
\end{equation}
for any deformation $u_t$ of $u$. Therefore, any maximal generalized conformal factor is extremal in the above sense.
\end{definition}

We remark that condition (\ref{comp}) is equivalent to
\begin{equation}
F_{2,\epsilon}(u_t) \le F_{2,\epsilon}(u) + o(t) \;\;\;\; (t\to 0),
\end{equation}
for any deformation $u_t$ of $u$. This characterization was used by Nadarashvili in \cite{Nadirashvilli} for Laplace eigenvalues. In the following proposition, we derived the Euler Lagrange Equation for maximal functions of the regularized functional $F_{2,\epsilon}$. This will be the key for getting uniform estimates in Section \ref{Estimates}.

\begin{proposition} \label{EulerProp}  Suppose $u \in  \mathcal{D}_\epsilon$ is maximal and normalized as in (\ref{unormal}).   Then there is a set of eigenfunctions $\{\phi^{\alpha} \}_{\alpha=1}^k$ associated to $\lambda_2(u)$, normalized by
\begin{equation}
\int_M (\phi^{\alpha})^2 u^{N-2}\;dv_g=1,
\end{equation}
and a set of real numbers $c_1,\dots,c_k \geq 0$ with $\sum_{\alpha=1}^k c_{\alpha}^2  = 1$, such that
\begin{align} \label{Euler}
\gamma_1 u^{N} -  u^{N-2} \sum_{\alpha=1}^k c_{\alpha} (\phi^{\alpha})^2 - \gamma_2  u^{-\epsilon} = 0,
\end{align}
where
\begin{align} \label{gammas} \begin{split}
\gamma_1 &= 1 - (N-2)^{-1} \frac{\epsilon}{\lambda_2(u)}   \int_M u^{-\epsilon} \, dv_g > 1, \\
\gamma_2 &= (N-2)^{-1} \frac{\epsilon}{ |\lambda_2(u)|} > 0.
\end{split}
\end{align}
\end{proposition}

\begin{proof}
Consider the subset $K\subseteq L^1$ defined by
\begin{equation}
K:= \left\{\phi^2\cdot u^{N-2}: \phi\in E_2(u),\|\phi\|_{L^2(u)}=1\right\}.
\end{equation}
This set is compact, and since $E_2(u)$ is finite dimensional, it lies in a finite dimensional subspace of $L^1$. Caratheodory's Theorem for Convex Hulls implies that the convex hull of $K$,
\begin{align*}
\text{Conv}(K) &= \left\{\sum_{\text{finite}}c_j\psi_j: c_j\ge 0, \sum c_j =1, \psi_j\in K\right\} \\ & =\left\{u^{N-2}\sum_{\text{finite}}c_j\phi_j^2: c_j\ge 0, \sum c_j =1, \phi_j\in E_2(u), \|\phi\|_{L^2(u)}=1\right\},
\end{align*}
is compact as well.

Notice that the proof will be completed if we can show that $\gamma_1u^{N} - \gamma_2u^{-\epsilon}\in \text{Conv}(K)$. Assume, on the contrary, that
\begin{equation}
\{\gamma_1u^{N} - \gamma_2u^{-\epsilon}\}\cap\text{Conv}(K)=\emptyset.
\end{equation}
This gives us two disjoint convex sets, the first one of them being closed, and the second one of them being compact. By Hahn-Banach Separation Theorem, there exists a functional $\Psi\in(L^1)^*$ separating these two sets:
\begin{equation}\label{HBS1}
\Psi(\gamma_1u^{N} - \gamma_2u^{-\epsilon}) > 0,
\end{equation}
and
\begin{equation}\label{HBS2}
\Psi(\varphi)\le 0\text{ for all }\varphi\in\text{Conv}(K).
\end{equation}
Furthermore, Ries'z Representation Theorem provides us with the existence of a function $h\in (L^1)^* = L^\infty$ so that $\Psi$ is given by integration against $h\; dv_g$. In particular, inequalities (\ref{HBS1}) and (\ref{HBS2}) can be rewritten as
\begin{equation}\label{HBS3}
\int_M h\cdot(\gamma_1u^{N} - \gamma_2u^{-\epsilon})\;dv_g > 0,
\end{equation}
and
\begin{equation}\label{HBS4}
\int_M h\cdot \phi^2u^{N-2}\;dv_g \le 0,
\end{equation}
where in (\ref{HBS4}) we have taken $\varphi\in\text{Conv}(K)$ to be a convex combination in $K$ of length one. Estimates (\ref{HBS3}) and (\ref{HBS4}) imply that for any $\phi\in E_2(u)$ with $\|\phi\|_{L^2(u)}=1$, we have
\begin{equation}\label{contradiction}
\int_M h(\gamma_1u^{N} - \gamma_2u^{-\epsilon})\;dv_g - \int_M h\phi^2u^{N-2}\;dv_g > 0.
\end{equation}

Now, we use $h$ as a generating function, that is, we consider $u_t = u(1+th)$, and select an interval $(-\delta,\delta)$ on which $1-|th|>0$. Since $h\in L^\infty$, Proposition \ref{OSDProp} holds, and the maximality of $u$ implies, in particular,
\begin{equation}
\frac{d}{dt} F_{2,\epsilon}(u_t) \big|_{t = 0^{-}} \ge 0.
\end{equation}
However, by (\ref{contradiction}),
\begin{align*}
\frac{d}{dt} F_{2,\epsilon}(u_t) \big|_{t = 0^{-}} & = (N-2)\lambda_2(u)\left\{ \int_M h(\gamma_1u^{N} - \gamma_2u^{-\epsilon})\;dv_g - \sup_{\phi \in E_2(u)}\int_M h\phi^2u^{N-2}\;dv_g\right\} \\ & <0.
\end{align*}
This is a contradiction, and hence $\gamma_1u^{N} - \gamma_2u^{-\epsilon}\in \text{Conv}(K)$, as claimed.
\end{proof}

In Proposition \ref{RegEx} we will show the existence of a maximizer for the regularized functional $F_{2,\epsilon}$.  We remark that one reason for introducing the regularizing term
\begin{align} \label{regtermiso}
u \longmapsto \int_M u^{-\epsilon}\;dv_g
\end{align}
is to prevent an extremal function (whose existence we will prove) from vanishing on a set of positive measure.  In addition, the functional (\ref{regtermiso}) is weakly lower semicontinuous on $\mathcal{D}_\epsilon$. In general, weak lower semicontinuity depends on some kind of convexity of the integrand; see Chapter 2 in \cite{Rindler}.

\begin{proposition}  \label{RegEx} For each $\epsilon > 0$, there is a $u_{\epsilon} \in \mathcal{D}_\epsilon$, normalized by
 \begin{align} \label{Vol}
 \int_M  u_{\epsilon}^N \, dv_g = 1,
 \end{align}
 that is maximal for $F_{2,\epsilon}$.  Moreover, there is a constant $C= C(g)$, independent of $\epsilon > 0$, such that
 \begin{align} \label{ue}
 \int_M u_{\epsilon}^{-\epsilon} \, dv_g \leq C.
 \end{align}

\end{proposition}

\begin{proof}  Fix $\epsilon > 0$, and let $\{ u_i \}_{i=1}^\infty \subset  \mathcal{D}_\epsilon$ be a normalized maximizing sequence for $F_{2,\epsilon}$, i.e. a sequence such that
\begin{align} \label{seq} \begin{split}
&\int_M u_i^N \, dv_g = 1, \\
&F_{2,\epsilon}(u_i) \rightarrow \sup_{ u \in \mathcal{D}_\epsilon } F_{2,\epsilon}(u).
\end{split}
\end{align}

Let $v_i = u_i^{N-2}$, then
\begin{align*}
 \int_M v_i^{\frac{n}{2}} \, dv_g = \int_M v_i^{\frac{N}{N-2}} \, dv_g = \int_M u_i^N \, dv_g = 1.
\end{align*}
In particular, the sequence $\{ v_i \}$ is bounded in $L^{\frac{n}{2}}(M)$ and, by weak compactness, there is a subsequence (still denoted $\{ v_i \}$) that converges weakly in $L^{\frac{n}{2}}(M)$ to some $\bar{v} \in L^{\frac{n}{2}}(M)$, with $\bar{v} \geq 0$ a.e. and
\begin{align} \label{bunv}
\int_M \bar{v}^{\frac{n}{2}} \, dv_g \leq 1.
\end{align}
Let
\begin{align} \label{bard}
\bar{u} = \bar{v}^{\frac{1}{N-2}},
\end{align}
then
\begin{align} \label{bun}
\int_M \bar{u}^N \, dv_g \leq 1.
\end{align}

By the second condition in (\ref{seq}), we may assume for $i >> 1$ large that $F_{2,\epsilon}(u_i) \geq F_{2,\epsilon}(1)$, and therefore
\begin{equation}\label{EigenBound}
\lambda_2(u_i) - \int_M u_i^{-\epsilon} \, dv_g \geq \lambda_2(L_g) - 1,
\end{equation}
where we have used the assumption $\text{Vol}(M,g) = 1$.  As $\lambda_2(u_i) < 0$, it follows that
\begin{align} \label{Leb1}
\int_M u_i^{-\epsilon}\;dv_g \leq C(g),
\end{align}
where $C$ is independent of $\epsilon$.  Therefore,
\begin{align} \label{Leb}
\int_M v_i^{-\epsilon'}\; dv_g \leq C(g),
\end{align}
where $\epsilon' = \epsilon/(N-2)$.  Since $f : (0,\infty) \rightarrow \mathbb{R}$ given by $f(t) = t^{-\epsilon'}$ is convex, the functional
\begin{align*}
v \longmapsto \int_M v^{-\epsilon'} \, dv_g
\end{align*}
is weakly lower semicontinuous on $L^{\frac{n}{2}}(M)$, thus
\begin{align} \label{wlsc}
\int_M \bar{u}^{-\epsilon} \, dv_g = \int_M \bar{v}^{-\epsilon'} \, dv_g \leq \liminf_{i \to \infty} \int_M v_i^{-\epsilon'} \, dv_g.
\end{align}
In particular, by (\ref{Leb}),
\begin{align} \label{ube}
\int_M \bar{u}^{-\epsilon} \, dv_g \leq C.
\end{align}
This shows that (\ref{ue}) holds for $\bar{u}$, hence $\bar{u} \in \mathcal{D}_\epsilon$.  In particular, $\bar{u} \in L^N_{>0}(M)$.

It remains to show that $\bar u$ is maximal.  By Proposition \ref{Cpropplus}, after possibly restricting to a subsequence,
\begin{align}  \label{lamlim}
\lambda_2(u_i) \rightarrow \lambda_2(\bar{u}) < 0.
\end{align}
Therefore, by (\ref{bun}),
\begin{align} \label{one}
\limsup_{i \to \infty} \lambda_2(u_i) = \lambda_2(\bar{u}) \leq \lambda_2(\bar{u}) \left( \int_M \bar{u}^N \, dv_g \right)^{\frac{N-2}{N}}.
\end{align}
Also, by (\ref{bun}) and (\ref{wlsc}),
\begin{align} \label{two}
\left( \int_M \bar{u}^N \, dv_g \right)^{\frac{\epsilon}{N}} \left( \int_M \bar{u}^{-\epsilon} \, dv_g \right) \leq \liminf_{i \to \infty} \int_M u_i^{-\epsilon} \, dv_g,
\end{align}
and thus
\begin{align} \label{twop}
- \left( \int_M \bar{u}^N \, dv_g \right)^{\frac{\epsilon}{N}} \left( \int_M \bar{u}^{-\epsilon} \, dv_g \right) \geq  \limsup_{i \to \infty}  \left\{-\int_M u_i^{-\epsilon} \, dv_g \right\}.
\end{align}
Combining (\ref{one}) and (\ref{twop}), we conclude
\begin{align*}
F_{2,\epsilon}(\bar{u}) &= \lambda_2(\bar{u}) \left( \int_M \bar{u}^N \, dv_g \right)^{\frac{N-2}{N}} - \left( \int_M \bar{u}^N \, dv_g \right)^{\frac{\epsilon}{N}} \left( \int_M \bar{u}^{-\epsilon} \, dv_g \right) \\
&\geq \limsup_{i \to \infty} \left\{  \lambda_2(u_i) -  \int_M u_i^{-\epsilon} \, dv_g \right\} \\
&= \sup_{ u \in \mathcal{D}_\epsilon } F_{2,\epsilon}(u),
\end{align*}
and it follows that $\bar{u}$ is maximal.  The estimate (\ref{ue}) follows from (\ref{ube}). This ends the proof.

\end{proof}

\vskip.2in
\section{Estimates}\label{Estimates}

Suppose that $u_{\epsilon} \in  \mathcal{D}_\epsilon$ is maximal for $F_{2,\epsilon}$, normalized as in (\ref{unormal}):
\begin{align*}
\int_M u_{\epsilon}^N \, dv_g = 1.
\end{align*}
By Proposition \ref{EulerProp} this means that there is a set of second generalized eigenfunctions $\{ \phi^{\alpha}_{\epsilon}\}_{\alpha=1}^{k(u_\epsilon)} $ associated to $\lambda_2(u_{\epsilon})$, normalized so that
\begin{align} \label{phiL2}
\int_M (\phi^{\alpha}_{\epsilon})^2 u_{\epsilon}^{N-2} \, dv_g = 1,
\end{align}
and a set of real numbers $c_{1,\epsilon},\dots,c_{k(u_\epsilon),\epsilon} \geq 0$ with $\sum_{\alpha=1}^{k(u_\epsilon)} c_{\alpha,\epsilon}  = 1$, for which
\begin{align} \label{Euler}
\gamma_{1,\epsilon} u_{\epsilon}^N -  u_{\epsilon}^{N-2} \sum_{\alpha=1}^{k(u_\epsilon)} c_{\alpha,\epsilon} (\phi^{\alpha}_{\epsilon})^2 - \gamma_{2,\epsilon}  u_{\epsilon}^{-\epsilon} = 0
\end{align}
holds. Here $\gamma_{1,\epsilon}$ and $\gamma_{2,\epsilon}$ denote the real numbers
\begin{align} \label{gammas} \begin{split}
\gamma_{1,\epsilon} &= 1 - (N-2)^{-1} \frac{\epsilon}{\lambda_2(u_{\epsilon})}   \int_M u_{\epsilon}^{-\epsilon} \, dv_g > 1, \\
\gamma_{2,\epsilon} &= (N-2)^{-1} \frac{\epsilon}{ |\lambda_2(u_{\epsilon})|} > 0.
\end{split}
\end{align}

Our goal is to take the limit $\epsilon\to 0^+$, to obtain an extremal function for the eigenvalue functional $F_2$.  In this section we prove various {\em a priori} estimates that will be used to prove the convergence of a sequence of maximal metrics in Section \ref{TakingLimit}.  To this end, let $\epsilon_i \rightarrow 0$ be an arbitrary sequence, and $u_i = u_{\epsilon_i}$ be a sequence of corresponding maximal generalized conformal factors.  By restricting to a subsequence we may assume that $v_i = u_i^{N-2}$ converges weakly in $L^{n/2}(M)$.

Let $m(u_i)$ denote the multiplicity of $\lambda_2(u_i)$, that is, the dimension of $E_2(u_i)$. By Proposition \ref{NegativeEigen},
\begin{equation}
1\le k(u_i) \le m(u_i) \le \nu([g]).
\end{equation}
By further restricting to a subsequence, we may assume $k(u_i) = k$, $m(u_i) = m$ for all $i \geq 1$.

Furthermore, since $F_{2,\epsilon_i}(u_i)\ge F_{2,\epsilon_i}(1)$, it follows that
\begin{equation} \label{EV1}
0 < |\lambda_2(u_i)| \le |\lambda_2(L_g) - 1|.
\end{equation}
Since we have a uniform bound on $\lambda_2(u_i)$, by Proposition \ref{Cprop} we have uniform bounds for the sequence of second generalized eigenfunctions $\{ \phi_{i,l} \}$:
\begin{align} \label{W12p}
\| \phi^{\alpha}_i\|_{W^{1,2}(M)} + \| \phi^{\alpha}_i \|_{L^{\infty}(M)} \leq C.
\end{align}
Moreover, as $i \to \infty$, $\{ \phi^{\alpha}_i \}$ converges weakly in $W^{1,2}(M)$ and strongly in $L^{N/2}(M)$ to $\phi^{\alpha} \in W^{1,2} \cap L^{\infty}$, a $W^{1,2}$-solution of
\begin{align} \label{faible}
L_g \phi^{\alpha} = \bar{\lambda}_2 \phi^{\alpha} \bar{u}^{N-2},
\end{align}
where $\lambda_2(u_i) \rightarrow \bar{\lambda}_2$.  Also, $\phi^{\alpha}$ satisfies the normalization
\begin{align*}
\int_M (\phi^{\alpha})^2 \bar{u}^{N-2} \, dv_g = 1.
\end{align*}

Note that $\bar{\lambda}_2 < 0$; otherwise (\ref{faible}) would imply that $\phi^{\alpha} \in \ker L_g$.  Consequently, we may assume
\begin{align} \label{llooww}
\lambda_2(u_i) \leq \lambda_0 < 0.
\end{align}

Recall that by Proposition \ref{RegEx}, the sequence $\{u_i \}_{\epsilon>0}$ satisfies
\begin{align} \label{ube2}
\int_M u_i^{-\epsilon_i} \, dv_g \leq C,
\end{align}
where $C$ is independent of $i$.  This, together with (\ref{llooww}), gives us the estimates
\begin{align} \label{gams}   \begin{split}
1 &\leq \gamma_{1,\epsilon_i} \leq c_1, \\
c_2^{-1} \epsilon_i &\leq \gamma_{2,\epsilon_i} \leq c_2 \epsilon_i,
\end{split}
\end{align}
for constants $c_1, c_2 > 1$ that only depend on the background metric $g$.


\begin{lemma}\label{u-upper} There is a constant $C = C(g)$, independent of $\epsilon$, such that
\begin{align} \label{sup}
\emph{ess sup } u_i \leq  C.
\end{align}
\end{lemma}

\begin{proof}  Let $x \in M$ such that $u_i(x) \geq 1$.  Then at $x$, the Euler equation (\ref{Euler}) implies
\begin{align*}
\gamma_{1,\epsilon_i} u_i^N(x) &=  u_i^{N-2}(x) \sum_{\alpha =1}^k c_{\alpha,\epsilon_i} \phi^{\alpha}_i(x)^2 + \gamma_{2,\epsilon_i}  u_i^{-\epsilon_i}(x) \\
&\leq u_i^{N-2}(x) \sum_{\alpha=1}^k c_{\alpha,\epsilon_i} \phi^{\alpha}_i(x)^2 + \gamma_{2,\epsilon_i}  \ \ \ \ \mbox{(since $u_i(x) \geq 1$)} \\
&\leq C (u_i^{N-2}(x) + \epsilon_i). \ \ \ \ \mbox{ (by (\ref{W12p}) and (\ref{gams}))}
\end{align*}
Therefore, $u_i(x) \leq C$, and we are done.
\end{proof}

As an immediate consequence of Lemma \ref{u-upper}, we have
\begin{corollary} \label{efCor} For any $\mu \in (0,1)$ and $p > 1$, there is a $C = C(\mu, p)$ such that
\begin{align} \label{Hold}
\| \phi_{i,l} \|_{C^{1,\mu}(M^n)} + \| \phi_{i,l}\|_{W^{2,p}(M^n)}  \leq C.
\end{align}
\end{corollary}

\begin{proof}
Since $\phi^{\alpha}_i$ is a second generalized eigenfunction, it is a $W^{1,2}$-solution of
\begin{align} \label{eR}
L_g \phi^{\alpha}_i = \lambda_2(u_i) \phi^{\alpha}_i u_i^{N-2}.
\end{align}
The $L^{\infty}$-estimates on $\phi^{\alpha}_i$ (see (\ref{W12p})) and $u_i$ in Lemma \ref{u-upper}, along with the bound (\ref{EV1}) for $\lambda_2(u_i)$, imply that the right-hand side of (\ref{eR}) is uniformly bounded.  Then  the $W^{2,p}$-norm, for any $p>1$, follows from standard elliptic regularity theory.   The $C^{1,\mu}$-bound then follows from the Sobolev embedding theorem.
\end{proof}

While Lemma \ref{u-upper} gives uniform upper bounds for $u_i$, we can only obtain $\epsilon$-dependent lower bounds:

\begin{lemma} \label{u-lower} There is a constant $C = C(g)$ such that
\begin{align} \label{inf}
\emph{ess inf } u_i \geq (C \epsilon_i)^{\frac{1}{\epsilon_i +N}}.
\end{align}
\end{lemma}

\begin{proof}  Let $x \in M$ be such that $u_i(x)>0$.  Then at $x$, the Euler equation (\ref{Euler}) implies
\begin{align*}
\gamma_{1,\epsilon_i} u_i^N(x)
&=  u_i^{N-2}(x) \sum_{\alpha=1}^k c_{\alpha,\epsilon_i} \phi^{\alpha}_i(x)^2 + \gamma_{2,\epsilon_i}  u_i^{-\epsilon_i}(x) \\
&\geq  \gamma_{2,\epsilon_i}  u_i^{-\epsilon_i}(x) \\
&\geq C \epsilon_i u_i^{-\epsilon_i}(x).
\end{align*}
Since $\gamma_{1,\epsilon_i} \leq C$ by (\ref{gams}), we conclude
\begin{align*}
u_i^{\epsilon_i + N}(x) \geq C \epsilon_i,
\end{align*}
as claimed.
\end{proof}

Combining the preceding results, we now show that $u_i \in C^{1,\mu}(M^n)$.  To see this, divide through the Euler equation (\ref{Euler}) by $u_i^{N-2}$ and rearrange terms to obtain
\begin{align} \label{NewEuler}
\gamma_{1,\epsilon_i} u_i^2 - \gamma_{2,\epsilon_i}  u_i^{-\epsilon_i - N + 2} = \Phi_i,
\end{align}
where
 \begin{align} \label{Phi}
 \Phi_i = \sum_{\alpha=1}^k c_{\alpha,\epsilon_i} (\phi^{\alpha}_i)^2.
 \end{align}
 If we define $f_i : (0,\infty) \rightarrow \mathbb{R}$ by
 \begin{align*}
 f_i(t) := \gamma_{1,\epsilon_i} t^2 -\gamma_{2,\epsilon_i} t^{-\epsilon_i - N + 2},
 \end{align*}
 then $f_i \in C^{\infty}(\mathbb{R}^{+})$, and by (\ref{NewEuler}),
 \begin{align} \label{new2}
 f(u_i) = \Phi_i.
 \end{align}
Since $f_i$ is strictly increasing on $(0,\infty)$, its inverse exists and we can write
 \begin{align*}
 u_i = f_i^{-1} \circ \Phi_i.
 \end{align*}
This means that $u_i$ can be written as the composition of two $C^{1,\mu}$-functions.  Note that we cannot obtain uniform estimates for the $C^{1,\mu}$ norm of $u_i$, since the derivative of $f_i^{-1}(t)$ can be quite large when $t$ is small; that is, on a set where $\Phi_i$ is small, the $C^{1,\mu}$-norm of $u_i$ will be large.   However, we can prove that $u_i$ has a uniform Lipschitz bound:

\begin{lemma} \label{C1Lemma}  There is constant $C = C(g)$, independent of $\epsilon$, such that
\begin{align} \label{C12}
|\nabla_g u_i| \leq C.
\end{align}
\end{lemma}

\begin{proof}  Differentiating both sides of (\ref{new2}) we find
\begin{align} \label{new3}
f'(u_i)\nabla_g u_i = \nabla_g \Phi_i,
\end{align}
and, therefore,
\begin{equation}
f_i'(u_i) |\nabla_g u_i| = |\nabla_g \Phi_i|.
\end{equation}
Now, from (\ref{Phi}) and Cauchy-Schwarz inequality,
\begin{equation}
|\nabla_g \Phi_i|\le 2\Phi_i^{\frac{1}{2}}\left(\sum_{\alpha=1}^k|\nabla_g\phi^{\alpha}_i|^2\right)^{\frac{1}{2}} \le C \Phi_i^{\frac{1}{2}},
\end{equation}
where the last inequality follows from Corollary \ref{efCor}. Using (\ref{NewEuler}) we deduce
\begin{equation}
|\nabla_g\Phi_i|\le C u_i.
\end{equation}
On the other hand, for $t > 0$,
\begin{align*}
f_i'(t) = 2 \gamma_{1,\epsilon_i} t + \gamma_{2,\epsilon_i} (\epsilon_i + N - 2) t^{-\epsilon_i - N + 1} \geq 2 t,
\end{align*}
and therefore
\begin{align*}
2u_i |\nabla_g u_i| \le f_i'(u_i)|\nabla_g u_i| = |\nabla_g \Phi_i| \le C u_i.
\end{align*}
This finishes the proof.
\end{proof}

 The following lemma is needed in the proof of Corollary \ref{keycor} below, and will also be key when taking the limit $\epsilon \to 0^+$ in the next section.

\begin{lemma} \label{Int2}
\begin{align} \label{int2}
\epsilon_i \int_M u_i^{-\epsilon_i - N} \, dv_g \leq C,
\end{align}
where $C$ only depends on $g$.
\end{lemma}

\begin{proof} If we divide both sides of (\ref{Euler}) by $u_i^N$ and integrate, we get
\begin{align*}
 \gamma_{2,\epsilon_i} \int_M  u_i^{-\epsilon_i- N} \, dv_g = \gamma_{1,\epsilon_i} \int_M \, dv_g -  \int_M u_i^{-2} \sum_{\alpha=1}^k c_{\alpha,\epsilon_i} (\phi^{\alpha}_i)^2 \, dv_g  \leq \gamma_{1,\epsilon_i}  \leq C.
\end{align*}
Since $\gamma_{2,\epsilon_i} \geq c_2^{-1} \epsilon_i$ by (\ref{gams}), the result follows.
\end{proof}

\begin{corollary}  \label{keycor}  Let $\eta \in C^{\infty}(M^n)$.  Then
\begin{align} \label{asy}
\int_M \eta \left\{ \gamma_{1,\epsilon_i} u_i^2 - \sum_{\alpha=1}^k c_{\alpha,\epsilon} (\phi^{\alpha}_i)^2 \right\} \, dv_g = O(\epsilon^{\beta_i}) \| \eta \|_{L^{\infty}(M^n)},  \ \ \  i \to \infty,
\end{align}
where
\begin{align} \label{beta}
\beta_i:= \frac{2}{\epsilon_i + N} > 0.
\end{align}
Moreover,
\begin{align} \label{gammalim}
\gamma_{1,\epsilon_i} = 1 + O(\epsilon_i), \ \ i \to \infty.
\end{align}
\end{corollary}

\begin{proof} We multiply through the Euler equation (\ref{Euler}) by $\eta u_i^{2-N}$ to get
\begin{align} \label{f1}
\int_M \eta \left\{ \gamma_{1,\epsilon_i} u_i^2 - \sum_{\alpha=1}^k c_{\alpha,\epsilon_i} (\phi^{\alpha}_i)^2 \right\} \, dv_g = \gamma_{2,\epsilon_i} \int_M \eta u_i^{- \epsilon_i - (N-2)} \, dv_g.
\end{align}
Now, the term on the right can be estimated via H\"older's inequality as follows:
\begin{align*}
\left| \gamma_{2,\epsilon_i} \int_M \eta u_i^{2 - (\epsilon_i + N)} \, dv_g \right| &\leq \gamma_{2,\epsilon_i} \| \eta \|_{L^{\infty}} \int_M u_i^{- \epsilon_i - (N-2)} \, dv_g \\
&\leq C \epsilon_i  \| \eta \|_{L^{\infty}} \left( \int_M u_i^{- \epsilon_i - N} \, dv_g \right)^{\frac{ \epsilon_i + N - 2}{\epsilon_i + N}} \left( \int_M \, dv_g \right)^{\frac{2}{\epsilon_i + N}} \\
&= C  \epsilon^{\frac{2}{\epsilon_i + N}} \| \eta \|_{L^{\infty}}\left( \epsilon_i \int_M u_i^{- \epsilon_i - N} \, dv_g \right)^{\frac{ \epsilon_i + N - 2}{\epsilon_i + N}}.
\end{align*}
To finish the proof of the first part of the statement, note that the integral in parentheses above is bounded by Lemma \ref{Int2}. Therefore,
\begin{align*}
\left| \gamma_{2,\epsilon_i} \int_M \eta u_i^{2 - (\epsilon_i + N)} \, dv_g \right| \leq C \epsilon_i^{\beta_i}  \| \eta \|_{L^{\infty}},
\end{align*}
where $\beta_i$ is given by (\ref{beta}).  The estimate for $\gamma_{1,\epsilon_i}$ is immediate from (\ref{ube2}) and (\ref{gams}).
\end{proof}

\vskip.2in
\section{Taking the limit $\epsilon \rightarrow 0$}\label{TakingLimit}

As in the previous section, let $\epsilon_i \rightarrow 0$, and $u_i = u_{\epsilon_i}$ be maximal for $F_{2,\epsilon_i}$.  By Lemma \ref{C1Lemma} we may assume that, along some subsequence,
\begin{align}  \label{usubc}
u_i \rightarrow \bar{u} \ \ \mbox{in } C^{\mu_0}(M^n),
\end{align}
for some $\mu_0\in(0,1)$. Since $\| u_i \|_{L^N} = 1$, it follows that
\begin{align} \label{VB}
\int_M \bar{u}^N \, dv_g = 1.
\end{align}
If we let $\{ \phi^{\alpha}_i  \}_{\alpha=1}^k$ be the corresponding set of eigenfunctions associated to $\lambda_2(u_i)$,  then each eigenfunction satisfies
\begin{align} \label{efef}
L_g \phi^{\alpha}_i = \lambda_2(u_i) \phi^{\alpha}_i u_i^{N-2},
\end{align}
with the normalization
\begin{align} \label{normlast}
\int_M (\phi^{\alpha}_i)^2 u_i^{N-2} \, dv_g = 1.
\end{align}
Thanks to Corollary \ref{efCor}, by further restricting our subsequence if necessary, we may assume
\begin{equation} \label{eflim}
\phi^{\alpha}_i \rightarrow \bar{\phi}^{\alpha} \ \ \mbox{in } C^{1,\mu_0}(M^n),
\end{equation}
and $\bar{\phi}^{\alpha}$ satisfies
\begin{align} \label{bpPDE}
L_g \bar{\phi}^{\alpha} = \bar{\lambda_2} \bar{\phi}^{\alpha} \bar{u}^{N-2},
\end{align}
where $\bar{\lambda}_2 = \lim_{i \to \infty} \lambda_2(u_i).$  By elliptic regularity and (\ref{usubc}), we see that $\bar{\phi}^{\alpha} \in C^{2,\mu_0}(M)$.

Furthermore, by Corollary \ref{keycor}, for any $\eta \in C^{\infty}(M^n)$ we have
\begin{equation} \label{asy2}
\lim_{i \to \infty} \int \eta \left\{ \gamma_{1,\epsilon_i} u_i^2 - \sum_{\alpha=1}^k c_{\alpha,\epsilon_i} (\phi^{\alpha}_i)^2 \right\} \, dv_g = 0.
\end{equation}
Again restricting to subsequences when necessary, we can assume $c_{\alpha,\epsilon_i} \rightarrow \bar{c}_{\alpha}$, and by (\ref{asy2}) we have
\begin{align} \label{NewEuler2}
\bar{u}^2 - \sum_{\alpha=1}^k \bar{c}_{\alpha} (\bar{\phi}^{\alpha})^2 = 0.
\end{align}
Note that some of the $\bar{c}_{\alpha}$'s could be zero, but not all can vanish since $\sum_{\alpha=1}^kc_{\alpha,\epsilon_i} = 1$.

We claim that $\bar{u} \in L^N_{>0}(M)$.  For, if $\bar{u}$ vanished on a set of positive measure $E$, then (\ref{NewEuler2}) would imply that $\bar{\phi}^{\alpha} = 0$ on $E$ whenever $\bar{c}_{\alpha} \neq 0$.  By unique continuation, $\bar{\phi}^{\alpha} \equiv 0$ on $M$, a contradiction.  In fact, locally the zero locus $\{ \bar{u} = 0 \}$ can be decomposed into a (possibly empty) $C^1$-submanifold of dimension $(n-1)$ with finite volume given by $\{ \bar{u} = 0 \} \cap \{ |\nabla \bar{u}| > 0 \}$, and a closed set  of the form $\{ \bar{u} = 0 \} \cap \{ |\nabla \bar{u}| = 0 \}$ with dimension $\leq n-2$ (see \cite{HS}, and references therein).  In particular, $\bar{u} \in L^N_{>0}(M)$, and by Proposition \ref{Cpropplus}, $\bar{\lambda}_2 = \lambda_2(\bar{u})$, and each $\bar{\phi}^{\alpha} \in E_2(\bar u)$.

What remains to be discussed, in order to finish the proof of Theorem \ref{MainTheorem}, is the maximality of $\bar u$ and its regularity outside its zero set. We address these aspects below.


\begin{proof}[Proof of Theorem \ref{MainTheorem}]
 Let $\bar{u} \in L^N_{>0}$ be as constructed above, and notice that $\|\bar u\|_{L^N}$.  The proof is by contradiction: if $\bar{u}$ is not maximal, then there is a
$w \in L^N_{>0}$ with $\| w \|_{L^N} = 1$ and
\begin{align*}
\lambda_2(w) = \lambda_2(\bar{u}) + \eta_0
\end{align*}
for some $\eta_0 > 0$ small.

Let $w_{\delta} := \sup \{ w , \delta \}$ for $0<\delta<1$. We claim that we can fix $\delta > 0$ small enough such that
\begin{align} \label{claim1}
 \lambda_2(w_{\delta}) \left( \int_M w_{\delta}^N\;dv_g \right)^{\frac{N-2}{N}} \geq \lambda_2(\bar{u}) + \frac{1}{2} \eta_0.
\end{align}
To prove (\ref{claim1}), we first observe that since $w_{\delta} \geq w$, then by comparing the Rayleigh quotient we get
\begin{equation}
\mathcal{R}_g^{w_\delta}(\phi)\ge \mathcal{R}_g^{w}(\phi)
\end{equation}
as we are allow to take as test functions those $\phi\in W^{1,2}$ for which
\begin{equation}
\int_M\{|\nabla_g\phi|^2+c_nR_g\phi^2\}\;dv_g <0.
\end{equation}
This is because $\nu(w_\delta) = \nu(w) = \nu([g])$ by Proposition \ref{NegativeEigen}. The min-max characterization then yields
\begin{align*}
\lambda_2(w_{\delta}) \geq \lambda_2(w).
\end{align*}
As for the volume factor, we have
\begin{align*}
\int_M w_{\delta}^N\;dv_g &= \int_{\{ w \leq \delta\} } w_{\delta}^N\;dv_g + \int_{\{w > \delta\}} w_{\delta}^N\;dv_g
= \int_{\{w \leq \delta\}} \delta^N\;dv_g + \int_{\{w > \delta\}} w^N\;dv_g \\
& \leq \delta^N + 1.
\end{align*}
Since $( 1 + \delta^N)^{\frac{N-2}{N}} \leq 1 + a_n \delta^N$ for some $a_n > 0$, it follows that
\begin{align*}
\left( \int_M w_{\delta}^N\;dv_g \right)^{\frac{N-2}{N}} \leq 1 + a_n \delta^N.
\end{align*}
Therefore,
\begin{align*}
 \lambda_2(w_{\delta}) \left( \int_M w_{\delta}^N\;dv_g \right)^{\frac{N-2}{N}} &\geq ( 1 + a_n \delta^N) \big( \lambda_2(\bar{u}) + \eta_0 \big) \\
 &\geq \lambda_2(\bar{u}) + \eta_0 - a_n \delta^N |\lambda_2(\bar{u})|.
\end{align*}
Inequality (\ref{claim1}) follows after choosing $\delta=\delta(|\lambda_2(\bar u)|)>0$ small enough.

Notice that $w_\delta$ is in $\mathcal{D}_\epsilon$ for all $\epsilon > 0$.  By (\ref{claim1}),
\begin{align} \label{Fw} \begin{split}
F_{2,\epsilon}(u_{\epsilon}) &\geq F_{2,\epsilon}(w_{\delta}) \\
&=  \lambda_2(w_{\delta}) \left( \int_M w_{\delta}^N\;dv_g \right)^{\frac{N-2}{N}} - \left( \int_M w_{\delta}^{-\epsilon} \;dv_g\right) \left( \int_M w_{\delta}^N\;dv_g \right)^{\frac{\epsilon}{N}} \\
&\geq \lambda_2(\bar{u}) + \frac{1}{2} \eta_0 - \left(\int_M w_{\delta}^{-\epsilon}\;dv_g \right) \left( \int_M w_{\delta}^N\;dv_g \right)^{\frac{\epsilon}{N}}.
\end{split}
\end{align}
Let $u_i = u_{\epsilon_i}$ be our sequence as before; then we know
\begin{equation}\label{Fep}
\limsup_{\epsilon_i \to 0} \lambda_2(u_i) = \lambda_2(\bar{u}).
\end{equation}
We also claim that
\begin{equation}\label{Fep2}
\lim_{i \to 0} \int u_i^{-\epsilon_i}\;dv_g \rightarrow 1.
\end{equation}
To see why, we use Lemma \ref{Int2}:
\begin{align*}
\int_M u_i^{-\epsilon_i}\;dv_g &\leq \left( \int_M u_i^{-N-\epsilon_i}\;dv_g \right)^{\frac{\epsilon_i}{\epsilon_i + N}} \text{Vol}(M^n,g)^{\frac{N}{N + \epsilon_i}} \\
&\leq \left( C \epsilon_i^{-1} \right)^{\frac{\epsilon_i}{\epsilon_i + N}}   \ \ \ \ \mbox{ (since $\text{Vol}(M^n,g) = 1$)} \\
&= C^{\frac{\epsilon_i}{\epsilon_i + N}} \epsilon_i^{\frac{- \epsilon_i}{\epsilon_i + N}} \\
&\rightarrow 1, \ \ i \to \infty.
\end{align*}
Therefore,
\begin{align} \label{lequ}
\lim_{i \to\infty} \int_M u_i^{-\epsilon_i}\;dv_g \leq 1.
\end{align}
 Also,
\begin{align*}
1 = \int_M 1 \, dv_g = \int_M u_i^{-\frac{\epsilon_i}{2}} u_i^{\frac{\epsilon_i}{2}} \;dv_g \leq \left( \int_M u_i^{-\epsilon_i}\;dv_g \right)^{\frac{1}{2}} \left( \int_M u_i^{\epsilon_i}\;dv_g \right)^{\frac{1}{2}},
\end{align*}
and since $u_i$ is uniformly bounded, this implies
\begin{align} \label{gequ}
\lim_{i\to \infty}\int_M u_i^{-\epsilon_i}\;dv_g \geq 1.
\end{align}
Formula (\ref{Fep2}) follows then from (\ref{lequ}) and (\ref{gequ}).

An immediate consequence of (\ref{Fep}) and (\ref{Fep2}) is that
\begin{align} \label{conc}
F_{2, \epsilon_i}(u_i) \rightarrow \lambda_2(\bar{u}) - 1.
\end{align}
Recall that by (\ref{Fw}),
\begin{align} \label{limF2}
F_{2,\epsilon_i}(u_i) \geq \lambda_2(\bar{u}) + \frac{1}{2} \eta_0 - \left( \int_M w_{\delta}^{-\epsilon_i}\;dv_g \right) \left( \int_M w_{\delta}^N\;dv_g \right)^{\frac{\epsilon_i}{N}}.
\end{align}
Taking the limit as $\epsilon_i \rightarrow 0$, using the fact that $w_{\delta} \geq \delta > 0$, and plugging in (\ref{conc}), we conclude
\begin{align}
\lambda_2(\bar{u}) - 1 \geq \lambda_2(\bar{u}) + \frac{1}{2} \eta_0 - 1.
\end{align}
This is a contradiction.

The regularity of $\bar u$ outside its zero set follows by the fact that each $\bar \phi_i$ is in $C^{2,\mu_0}(M^n)$ and a standard bootstrap argument. This finishes the proof of our main result.
\end{proof}

\medskip

\begin{proof}[Proof of Corollary \ref{NodalHarmonic}]
Let $\bar u\in C^{\mu_0}(M^n)\cap C^\infty(M^n\setminus\{\bar u=0\})$ be the maximal function given by Theorem \ref{MainTheorem}. In the case where $k=1$, from(\ref{NewEuler2}) we see that $\bar u = |\bar \phi|$ for some $\bar \phi \in E_2(\bar u)\cap C^{2,\mu_0}(M^n)$, hence by (\ref{bpPDE}) we have
\begin{equation}
L_g\bar\phi = \lambda_2(\bar u) \bar \phi |\bar \phi|^{N-2}.
\end{equation}
If $\phi_1 \in E_1(\bar{u})$ denotes a first generalized eigenfunction, then we may assume $\phi_1 \geq 0$ a.e.  Also, by Proposition \ref{efProp},
\begin{align*}
\int_M \phi_1 \phi \bar{u}^{N-2} \, dv_g =0.
\end{align*}
It follows that $\phi$ must change sign; i.e., it is a nodal solution of the Yamabe equation.

We now assume that $k>1$.  On the open set $M^n\setminus\{\bar u =0\}$, $g_{\bar u}:= \bar u^{N-2}g$ is a smooth Riemannian metric. If we define $\psi^\alpha$ by
\begin{equation}
\psi^\alpha := \frac{\bar c_\alpha^{\frac{1}{2}}\bar\phi^\alpha}{\bar u},
\end{equation}
then from the Euler equation (\ref{NewEuler2}) we get
\begin{equation}\label{NewEuler3}
1 = \sum_{\alpha=1}^k (\psi^\alpha)^2.
\end{equation}
Also, by conformal invariance (\ref{Conf-Invar}), these new functions satisfy
\begin{equation}\label{NewEigenvalueEq}
L_{g_{\bar u}}(\psi^\alpha) = \lambda_2(\bar u) \psi^\alpha,
\end{equation}
for each $\alpha\in\{1,\cdots, k\}$.

We apply $L_{g_{\bar u}}$ on both sides of (\ref{NewEuler3}) to obtain
\begin{equation}
\begin{split}
c_nR_{g_{\bar u}} & = L_{g_{\bar u}}(1) =  \sum_{\alpha=1}^kL_{g_{\bar u}}((\psi^\alpha)^2) \\ & = \sum_{\alpha=1}^k\left(2\psi^\alpha L_{g_{\bar u}}\psi^\alpha - 2|\nabla_{g_{\bar u}}\psi^\alpha|^2 - c_nR_{g_{\bar u}}(\psi^\alpha)^2\right) \\ & = 2 \lambda_2(\bar u)\left(\sum_{\alpha=1}^k(\psi^\alpha)^2\right) - 2\sum_{\alpha=1}^k|\nabla_{g_{\bar u}}\psi^\alpha|^2 - c_nR_{g_{\bar u}}\left(\sum_{\alpha=1}^k(\psi^\alpha)^2\right) \\ & = 2\lambda_2(\bar u) - 2\sum_{\alpha=1}^k|\nabla_{g_{\bar u}}\psi^\alpha|^2 - c_nR_{g_{\bar u}},
\end{split}
\end{equation}
therefore,
\begin{equation}
\lambda_2(\bar u) = \sum_{\alpha=1}^k|\nabla_{g_{\bar u}}\psi^\alpha|^2 + c_n R_{g_{\bar u}}.
\end{equation}
Plugging this expression back into (\ref{NewEigenvalueEq}) gives us
\begin{equation}
-\Delta_{g_{\bar u}} \psi_\beta = \left(\sum_{\alpha=1}^k|\nabla_{g_{\bar u}}\psi^\alpha|^2\right)\psi_\beta
\end{equation}
for each $\beta\in\{1,\cdots,k\}$. Therefore, the map $U$ defined in (\ref{Harmonic}) is weakly harmonic; see \cite{Helein} and references therein for further details. Since $U$ is $C^2$ on $M^n\setminus\{\bar u=0\}$, the proof is complete.
\end{proof}

\begin{remark}\label{Simplicity} As a concluding remark we would like to explain a key difference between our setting and the setting of Ammann-Humbert in \cite{Ammann}. In their work, they assume the conformal class has non-negative Yamabe invariant, and prove the existence of a generalized conformal factor $u$ minimizing the second eigenvalue of the conformal laplacian.   Moreover, by a clever argument they show that the multiplicity of $\lambda_2(u)$ is always one.   We now explain how this fact follows from variational considerations.

Let $u_e\in L^N_{>0}$ be an extremal for $F_2$, normalized such that
\begin{equation}
\int_M u_e^N\; dv_g = 1.
\end{equation}
For any $h\in L^\infty$, we deform $u_e$ as usual by $u_{e,t} = u_e(1+th)$. Proposition \ref{OSD} still holds in this case and can be proven by similar techniques. Since $\lambda_2(u_\epsilon)>0$, $-(N-2)\lambda_2(u_e) <0$ and we obtain
\begin{equation}\label{min1}
\frac{d}{dt} F_2(u_{e,t}) \Big{|}_{t=0^+}= (N-2)\lambda_2(u_e)\left\{\int_M hu_e^N\;dv_g -\sup_{\phi}\int_M h\phi^2 u_e^{N-2}\;dv_g\right\},
\end{equation}
and
\begin{equation}\label{min2}
\frac{d}{dt} F_2(u_{e,t})\Big{|}_{t=0^-} = (N-2)\lambda_2(u_e)\left\{\int_M hu_e^N\;dv_g - \inf_{\phi}\int_M h\phi^2 u_e^{N-2}\;dv_g\right\},
\end{equation}
where the supremum and infimum are being taken over all $\phi \in E_2(u_e)$ with unit $L^2(u_e)$-norm.

In particular, if $u_e$ is minimal, meaning that
\begin{equation}
\frac{d}{dt} F_2(u_{e,t}) \Big{|}_{t=0^-} \le 0 \le \frac{d}{dt} F_2(u_{e,t})\Big{|}_{t=0^+}
\end{equation}
for any deformation $u_{e,t}$, then from (\ref{min1}), (\ref{min2}), and the fact that $\lambda_2(u_e)>0$, we get
\begin{equation}\label{min3}
\sup_{\phi}\int_M h\phi^2 u_e^{N-2}\;dv_g \le \inf_{\phi}\int_M h\phi^2 u_e^{N-2}\;dv_g.
\end{equation}
Therefore,
\begin{equation}\label{min4}
\sup_\phi \int_M h\phi^2u_e^{N-2}\;dv_g = \inf_\phi \int_M h\phi^2 u_e^{N-2}\;dv_g
\end{equation}
for all $h\in L^\infty$. If we are given two functions $\phi_1$ and $\phi_2$ in $E_2(u_e)$, normalized to have $L^2(u_e)$-unit norm, then (\ref{min4}) implies
\begin{equation}
\int_M h\phi_1^2u_e^{N-2}\;dv_g = \int_M h\phi_2^2u_e^{N-2}\;dv_g
\end{equation}
for every $h\in L^\infty$. This means the functions $\phi_1u_e^{\frac{N-2}{2}}$ and $\phi_2u_e^{\frac{N-2}{2}}$ are linearly dependent. Hence, the dimension of $E_2(u_e)$ is one.  

On the other hand, if we assume that $\nu([g])>1$ and that $u_e$ is maximal for $F_2$, as in Theorem \ref{MainTheorem}, then the argument above does not work: In fact, in Section \ref{Example} we construct an example of a maximal metric for which the second eigenvalue is degenerate. 
\end{remark}

\section{Example: $k>1$}\label{Example}

Let $(H,h)$ be a compact Riemannian manifold of dimension $n \geq 2$ with constant negative scalar curvature $R_H$.  By scaling down $h$ if necessary, we can assume that $(H,h)$  satisfies the following two properties:
\begin{enumerate}[(i)]
\item The scalar curvature
\begin{align} \label{RH}
R_H < - \frac{4n}{n-1},
\end{align}
\item The first non-zero eigenvalue $\lambda_1(-\Delta_h)$ satisfies
\begin{align} \label{Lbig}
\lambda_1(-\Delta_h) > 1.
\end{align}
\end{enumerate}

Let $(M,g) = (H \times S^1(1), h \oplus dt^2)$. Here $S^1=S^1(1)$ denotes the circle of radius $1$.  Then $(M,g)$ is a Riemannian manifold of dimension $m = n+1$ with the following properties:
\begin{enumerate}[(i)]
\item  The scalar curvature $R_g = R_H$ is constant.

\item If $L_g$ denotes the conformal laplacian, then
\begin{align} \label{Lform} \begin{split}
L_g &= -\Delta_g + \frac{(m-2)}{4(m-1)} R_g \\
&= -\Delta_g + \frac{n-1}{4n}R_H.
\end{split}
\end{align}
Consequently, the first eigenvalue of $L_g$ is $\lambda_1(L_g) = \frac{(n-1)}{4n}R_H$, with eigenfunctions given by constants.
\item  The {\em second} eigenvalue of $L_g$ is given by
\begin{align} \label{L2}
\lambda_2(L_g) = \lambda_1(-\Delta_{S^1(1)}) + \frac{(n-1)}{4n}R_H = 1 + \frac{(n-1)}{4n}R_H  < 0,
\end{align}
with eigenfunctions given by the first eigenfunction on the circle factors
\begin{align} \label{trigszed}
\psi_1 = \psi_1(t) = \cos t, \ \ \ \  \psi_2 = \psi_2(t) = \sin t.
\end{align}
In particular, $\lambda_2(L_g)$ has multiplicity two.
\end{enumerate}

\begin{theorem}\label{example}  The product metric $g$ is maximal in its conformal class.  In other words, if $\tilde{g} \in [g]$, then
\begin{align} \label{gmax}
\lambda_2(L_{\tilde{g}}) \text{Vol}(M,\tilde{g})^{\frac{2}{m}} \leq \lambda_2(L_g) \text{Vol}(M,g)^{\frac{2}{m}}.
\end{align}
Moreover, if equality holds then $u$ is constant.

Consequently, $(M,g) = (H \times S^1, h \oplus dt^2)$ is a maximal metric for which the eigenfunctions corresponding to $\lambda_2(g)$ define a harmonic map
\begin{align} \label{S1}
\Psi = (\psi_1, \psi_2) : M \rightarrow S^1,
\end{align}
given by projection onto the $S^1$-factor.
\end{theorem}

\medskip

\begin{remark}  As noted in the introduction, the argument below can be easily adapted to metrics on products of $H$ with a spheres of any dimension, producing examples
of maximal metrics for which $\lambda_2(L)$ has high multiplicity.
\end{remark}

\medskip

\begin{proof} Let $\tilde{g} = u^{\frac{4}{n-2}}g \in [g]$ with $u \in C^{\infty}(M)$ and $u > 0$.  We want to show
that
\begin{align} \label{goal1}
\lambda_2( \tilde{g}) \mbox{Vol}(\tilde{g})^{\frac{2}{m}} \leq \lambda_2(g) \text{Vol}(M,g)^{\frac{2}{m}}.
\end{align}
If we normalize $u$ so that
\begin{align} \label{bv}
\mbox{Vol}(M,\tilde{g}) = \int_M u^N \, dv_g = 1,
\end{align}
where $N = \frac{2m}{m-2}$, then (\ref{goal1}) is equivalent to
\begin{align} \label{goal2}
\lambda_2( \tilde{g}) \leq \lambda_2(g) \text{Vol}(M,g)^{\frac{2}{m}}.
\end{align}

Let $w_2$ be an eigenfunction associated to $\bl = \lambda_2(\tilde{g})$:
\begin{align*}
L_{\tilde{g}} w_2 = \bl w_2.
\end{align*}
By conformal invariance of $L$, the function
\begin{align} \label{p2}
\phi_2 = \dfrac{w_2}{u}
\end{align}
is an eigenfunction satisfying
\begin{align} \label{Lphi}
L_g \phi_2 = \bl \phi_2 u^{N-2}.
\end{align}
We normalize $\phi_2$ so that
\begin{align} \label{p2norm}
\int_M \phi_2^2 \, u^{N-2} \, dv_g = 1,
\end{align}
hence
\begin{align} \label{Ep}
\int_M \phi_2 \, L_g \phi_2 \, dv_g = \bl.
\end{align}
Likewise, let $\phi_1$ be the generalized eigenfunction associated to $\tf = \lambda_1(\tilde{g})$:
\begin{align} \label{Lphi1}
L_g \phi_1 = \tf  \phi_1 u^{N-2}.
\end{align}
We also normalize $\phi_1$ so that
\begin{align} \label{p1norm}
\int_M \phi_1^2 \, u^{N-2} \, dv_g = 1.
\end{align}
It follows from the strong maximum principle that $\phi_1 > 0$.

Let $t \in [0,2\pi)$ be the coordinate on $S^1$, then
\begin{align} \label{trigs}
\psi_1 = \psi_1(t) = \cos t, \ \ \ \  \psi_2 = \psi_2(t) = \sin t,
\end{align}
are first eigenfunctions for the laplacian on the $S^1$-factor.  Since these are also eigenfunctions for $\lambda_2(g)$, they satisfy
\begin{align} \label{Lpsi}
L_g \psi_i = \lambda_2(g) \psi_i, \ \ \ \ i = 1,2.
\end{align}
We also have the integral formulas
\begin{align} \label{Ints}
\int_M \psi_1^2 \, dv_g = \int_M \psi_2^2 \, dv_g = \text{Vol}(H,h)\cdot \pi = \frac{1}{2} \text{Vol}(M,g).
\end{align}
Since
\begin{align} \label{circle}
\psi_1^2 + \psi_2^2 = 1,
\end{align}
it follows that
\begin{align} \label{bal}
\int_M \left( \psi_1^2 + \psi_2^2 \right) \, u^{N-2} \, dv_g = \int_M u^{N-2} \, dv_g.
\end{align}
Therefore, we may assume (after relabeling if necessary) that $\psi_1$ satisfies
\begin{align} \label{half}
\int_M \psi_1^2  \, u^{N-2} \, dv_g \leq \frac{1}{2} \int_M u^{N-2} \, dv_g.
\end{align}

The following Lemma reduces the proof of the theorem to a key inequality:

\begin{lemma}  Suppose
\begin{align} \label{key}
\bl \leq \dfrac{ \frac{1}{2} \text{Vol}(M,g) \lambda_2(g)}{ \int_M \psi_1^2 \, u^{N-2} \, dv_g }.
\end{align}
Then (\ref{goal2}) holds. Moreover, if equality holds in (\ref{key}) then $u \equiv const.$ and equality holds in (\ref{goal2}).   \end{lemma}

\begin{proof}  If (\ref{key}) holds, then from (\ref{half}) it follows that
\begin{align*}
\frac{1}{2} \text{Vol}(M,g) \lambda_2(g) &\geq \bl \int_M \psi_1^2 \, u^{N-2} \, dv_g \geq \frac{1}{2} \bl \int_M u^{N-2} \, dv_g.
\end{align*}
By H\"older's inequality,
\begin{align} \label{Hold} \begin{split}
\frac{1}{2} \text{Vol}(M,g) \lambda_2(g) &\geq \frac{1}{2} \bl \int_M u^{N-2} \, dv_g \\
&\geq \frac{1}{2} \bl \left( \int_M u^N \, dv_g \right)^{\frac{2}{m}} \left( \int_M \, dv_g \right)^{\frac{m-2}{m}}\\ &= \frac{1}{2} \bl \text{Vol}(M,g)^{\frac{m-2}{m}},
\end{split}
\end{align}
and we conclude
\begin{align*}
\lambda_2(g) \, \text{Vol}(M,g)^{\frac{2}{m}} \geq \bl,
\end{align*}
and we see that (\ref{goal2}) holds.

If equality holds then we must have equality in H\"older's inequality in (\ref{Hold}), and $u \equiv const.$
\end{proof}

To prove (\ref{key}) holds, we will use the min-max characterization of $\lambda_2(L)$:
\begin{align} \label{mmax}
\bl = \inf_{\Sigma^2 \subset W^{1,2}(M,g) } \sup_{w \in \Sigma^2 \setminus \{ 0 \}} \dfrac{ E_g(w)}{\int w^2 u^{N-2} \, dv_g },
\end{align}
where $\Sigma^2 \subset W^{1,2}:=W^{1,2}(M,g)$ denotes a $2$-dimensional subspace of $W^{1,2}$ and
\begin{align} \label{Energy}
E_g(w) = \int_M w L_g w \, dv_g.
\end{align}
In particular, for any two-dimensional subspace $\Sigma^2 \subset W^{1,2}$,
\begin{align} \label{mmax}
\bl \leq \sup_{w \in \Sigma^2 \setminus \{ 0 \}} \dfrac{ E_g(w)}{\int_M w^2 u^{N-2} \, dv_g }.
\end{align}

Let
\begin{align*}
\Sigma^2 = \Sigma_0 = \mbox{span}\{ \psi_1,  \phi_2 \}.
\end{align*}
By homogeneity of the Rayleigh quotient,
\begin{align} \label{Rtrig0}
\sup_{w \in \Sigma_0 \setminus \{ 0 \}} \dfrac{ E_g(w)}{\int_M w^2 u^{N-2} \, dv_g } = \max_{ \theta \in [0,2\pi]} \dfrac{ E_g( \cos \theta \, \psi_1 + \sin \theta \, \phi_2) }{ \int_M ( \cos \theta \, \psi_1 +  \sin \theta \, \phi_2)^2 u^{N-2} \, dv_g },
\end{align}
hence
\begin{align} \label{Rtrig}
\bl \leq \max_{\theta \in [0,2\pi]} \dfrac{ E_g( \cos \theta \, \psi_1 + \sin \theta \, \phi_2) }{ \int_M ( \cos \theta \, \psi_1 + \sin \theta \, \phi_2)^2 u^{N-2} \, dv_g }.
\end{align}

By (\ref{Energy}),
\begin{align} \label{E} \begin{split}
E_g(& \cos \theta \, \psi_1 + \sin \theta \, \phi_2) \\
&= \int_M (\cos \theta \, \psi_1 + \sin \theta \, \phi_2) \, L_g (\cos \theta \, \psi_1 + \sin \theta \, \phi_2) \, dv_g \\
&= \cos^2 \theta \int_M \psi_1 L_g \psi_1 \, dv_g + \sin \theta \cos \theta \int_M \psi_1 \, L_g \phi_2 \, dv_g  \\
& \ \ \ \ +  \sin \theta \cos \theta \int_M \phi_2 \, L_g \psi_1 \, dv_g +  \sin^2 \theta \int_M \phi_2 L_g \phi_2 \, dv_g.
\end{split}
\end{align}
If we define
\begin{align} \label{alpha}
\alpha = \int_M \psi_1 \, \phi_2 \, u^{N-2} \, dv_g,
\end{align}
then since $L$ is self-adjoint,
\begin{align} \label{SA}  \begin{split}
\sin \theta \cos \theta \int_M \psi_1 \, L_g \phi_2 \, dv_g  &+  \sin \theta \cos \theta \int_M \phi_2 \, L_g \psi_1 \, dv_g \\
&= 2 \sin \theta \cos \theta \int_M \psi_1 \, L_g \phi_2 \, dv_g \\
&= 2 \bl \sin \theta \cos \theta  \int_M \psi_1 \, \phi_2 \, u^{N-2} \, dv_g \\
&= 2 \bl \alpha \, \sin \theta \cos \theta.
\end{split}
\end{align}
By (\ref{Ints}),
\begin{align} \label{E1}
\int_M \psi_1 L_g \psi_1 \, dv_g = \lambda_2(g) \int_M \psi_1^2 \, dv_g = \frac{1}{2} \text{Vol}(M,g) \lambda_2(g).
\end{align}
Substituting (\ref{SA}), (\ref{E1}), and (\ref{Ep}) into (\ref{E}), we obtain
\begin{align} \label{E}
E_g( \cos \theta \, \psi_1 + \sin \theta \, \phi_2) = \frac{1}{2} \text{Vol}(M,g) \lambda_2(g) \cos^2 \theta + \bl \Big[ 2 \alpha \sin \theta \cos \theta + \sin^2 \theta\Big].
\end{align}

Turning to the denominator in (\ref{Rtrig}), by the definition of $\alpha$ and the normalization of $\phi_2$,
\begin{align} \label{D1} \begin{split}
\int_M ( \cos \theta \, & \psi_1 +   \sin \theta \, \phi_2)^2 u^{N-2} \, dv_g \\
& = \int_M \big( \cos^2 \theta \, \psi_1^2 + 2  \sin \theta \cos \theta \psi_1 \phi_2 +  \sin^2 \theta \phi_2^2 \big) u^{N-2} \, dv_g \\
&= \cos^2 \theta \int_M \psi_1^2 u^{N-2} \, dv_g + 2  \sin \theta \cos \theta \int_M \psi_1 \phi_2 u^{N-2} \, dv_g +  \sin^2 \theta \int_M \phi_2^2 u^{N-2} \, dv_g \\
&= \cos^2 \theta \int_M \psi_1^2 u^{N-2} \, dv_g + \Big[ 2 \alpha \sin \theta \cos \theta + \sin^2 \theta\Big].
\end{split}
\end{align}
Therefore, the Rayleigh quotient of $w_{\theta} = (\cos \theta) \psi_1 + (\sin \theta) \phi_2$ is
\begin{align} \label{Rw1}
f(\theta) := \mathcal{R}[w_{\theta}] = \dfrac{ \frac{1}{2} \text{Vol}(M,g) \lambda_2(g) \cos^2 \theta + \bl \Big[ 2 \alpha \sin \theta \cos \theta + \sin^2 \theta\Big] }{ \cos^2 \theta \int_M \psi_1^2 u^{N-2} \, dv_g + \Big[ 2 \alpha \sin \theta \cos \theta + \sin^2 \theta\Big]}.
\end{align}
We need to consider two cases:  \medskip

\noindent {\bf Case 1.}  First, suppose $f$ attains its maximum at $\theta_c \in [0,2\pi]$ for which $\cos \theta_c = 0$.  Then
\begin{align} \label{case1}
\max_{ \theta \in [0, 2\pi] } f = \bl.
\end{align}
In this case, we need the following lemma:

\begin{lemma}  Suppose $\Sigma \subset W^{1,2}$ is a two-dimensional subspace such that
\begin{align} \label{sat}
\sup_{w \in \Sigma^2 \setminus \{ 0 \}} \dfrac{ E_g(w)}{\int w^2 u^{N-2} \, dv_g } = \bl.
\end{align}
Then
\begin{align} \label{split}
\Sigma \subset E_1 \oplus E_2,
\end{align}
where $E_1$ and $E_2$ are the spaces of (generalized) eigenfunctions corresponding to $\tf$ and $\bl$ respectively.
\end{lemma}

\begin{proof}  For classical eigenvalues this result is a simple consequence of the minimax principle.  As above let $\phi_1 \in E_1$ be a generalized eigenfunction associated to $\tf$, and
\begin{align*}
E_1^{\bot} = \{ v \in W^{1,2} \, : \, \int_M v \, \phi_1 \, u^{N-2} \, dv_g = 0 \}.
\end{align*}
Let $w_0 \in \Sigma \cap E_1^{\bot}$ with $w \neq 0$.  Such a $w_0$ must exist; otherwise, $\Sigma \subset E_1$, which contradicts the fact that $\Sigma$ is two-dimensional.  By the variational characterization of $\bl$,
\begin{align*}
\dfrac{ E_g(w_0)}{\int w_0^2 u^{N-2} \, dv_g } \geq \bl,
\end{align*}
with equality if and only if $w_0 \in E_2$.  However, by assumption,
\begin{align*}
\bl \geq \dfrac{ E_g(w_0)}{\int w_0^2 u^{N-2} \, dv_g }.
\end{align*}
It follows that $w_0 \in E_2$, and (\ref{split}) must hold.
\end{proof}

Applying the lemma to $\Sigma_0 = \{ \psi_1, \phi_2 \}$, we conclude that
\begin{align} \label{sub1}
\psi_1 = c_1 \phi_1 + c_2 \phi_2,
\end{align}
where $c_1, c_2$ are constants.  We claim that $c_1 = 0$.  If not, then since $\phi_1 > 0$, there cannot be a point $x_0 \in M$ at which
\begin{align*}
\psi_1(x_0) = \phi_2(x_0) = 0.
\end{align*}
Applying $L_g$ to both sides of (\ref{sub1}) and using the fact that $\phi_1$ and $\phi_2$ are generalized eigenfunctions for $\tf$ and $\bl$ respectively, we get
\begin{align} \label{Lk} \begin{split}
L_g \psi_1 &= c_1 L_g \phi_1 + c_2 L_g \phi_2 \\
&= c_1 \tf \phi_1 u^{N-2} + c_2 \bl \phi_2 u^{N-2} \\
&= \tf \left( \psi_1 - c_2 \phi_2 \right)u^{N-2} + c_2 \bl \phi_2 u^{N-2} \\
&= \tf \psi_1 u^{N-2} + c_2 \left( \bl - \tf \right) \phi_2 u^{N-2}.
\end{split}
\end{align}
Since $\psi_1$ is an eigenfunction for $\lambda_2(g)$, this implies
\begin{align} \label{Lk2}
\lambda_2(g) \psi_1 = \tf \psi_1 u^{N-2} + c_2 \left( \bl - \tf \right) \phi_2 u^{N-2}.
\end{align}
Let $x_0$ be a point at which $\psi_1(x_0) = 0$, then $\phi_2(x_0) \neq 0$.  But by (\ref{Lk2}),
\begin{align}
0 = c_2 \big( \bl - \tf \big) \phi_2(x_0) u(x_0)^{N-2},
\end{align}
hence $c_2 = 0$.  However, this would imply $\psi_1 = c_1 \phi_1$.  Since $\phi_1 > 0$ but $\psi_1$ obviously changes sign, we get a contradiction.  Therefore, $c_1 = 0$.

Since $c_1 = 0$, we have
\begin{align*}
\psi_1 = c_2 \phi_2.
\end{align*}
From this it immediately follows that
\begin{align*}
\frac{1}{2} \text{Vol}(M,g) \lambda_2(g) &= E_g(\psi_1) = c_2^2 E_g(\phi_2) = c_2^2 \bl,
\int_M \psi_1^2 \, u^{N-2} \, dv_g = c_2^2,
\end{align*}
hence
\begin{align*}
\dfrac{ \frac{1}{2} \text{Vol}(M,g) \lambda_2(g)}{ \int_M \psi_1^2 \, u^{N-2} \, dv_g } = \bl,
\end{align*}
so equality holds in (\ref{key}).  In particular, in this case equality holds in (\ref{goal2}).  \medskip

\noindent {\bf Case 2.}  The final case to consider is when the maximum of $f$ occurs at some $\theta \in [0,2\pi]$ for which $\cos \theta \neq 0$.  If $\cos \theta \neq 0$, we can rewrite $f$ as
\begin{align} \label{fdef}
f(\theta) = \dfrac{ \frac{1}{2} \text{Vol}(M,g) \lambda_2(g)  + \bl \left[ 2 \alpha \tan \theta + \tan^2 \theta \right] }{ \int_M \psi_1^2 u^{N-2} \, dv_g + \left[ 2 \alpha \tan \theta + \tan^2 \theta \right]},
\end{align}
with domain $(-\frac{\pi}{2}, \frac{\pi}{2})$.  If we take the derivative of $f$, we see that $\theta_c \in (-\frac{\pi}{2},\frac{\pi}{2})$ is a critical point of $f$ if and only if
\begin{align*}
2 \left( \bl \int_M \psi_1^2 \, u^{N-2} \, dv_g - \frac{1}{2} \text{Vol}(M,g) \lambda_2(g) \right) \left(\alpha + \tan \theta_c \right) \sec^2 \theta_c = 0.
\end{align*}
Therefore, we have two possibilities at the point where $f$ attains its maximum. First, if
\begin{align*}
\left( \bl \int_M \psi_1^2 \, u^{N-2} \, dv_g - \frac{1}{2} \text{Vol}(M,g) \lambda_2(g) \right) = 0,
\end{align*}
then equality holds in (\ref{key}) and we are done.  The second possibility is that $\tan \theta_c = -\alpha$.  Plugging this into the formula for $f$, we find
\begin{align} \label{pen}
\bl \leq \max f = \dfrac{ \frac{1}{2} \text{Vol}(M,g) \lambda_2(g) - \bl \alpha^2 }{ \int_M \psi_1^2 \, u^{N-2} \, dv_g - \alpha^2}.
\end{align}
Clearing the denominator gives
\begin{align} \label{pen2}
\bl \int_M \psi_1^2 \, u^{N-2} \, dv_g - \alpha^2 \bl \leq  \frac{1}{2} \text{Vol}(M,g) \lambda_2(g) - \bl \alpha^2,
\end{align}
hence
\begin{align} \label{ult}
\bl \int_M \psi_1^2 \, u^{N-2} \, dv_g \leq  \frac{1}{2} \text{Vol}(M,g) \lambda_2(g),
\end{align}
and once again (\ref{key}) holds.

Now suppose equality holds in (\ref{goal2}).  Then we have equality in (\ref{ult}), and therefore equality in (\ref{key}).  It follows that $u$ is constant.
\end{proof}


\end{document}